\documentclass[11pt]{amsart}
\usepackage{amsmath,amssymb,amsthm}

\usepackage{graphicx}
\usepackage{color}              % Need the color package
\usepackage{epsfig}
\usepackage{psfrag}
\usepackage{epstopdf}

\DeclareMathOperator{\twist}{tw}

\DeclareMathOperator{\X}{X}
\DeclareMathOperator{\h}{H}
\DeclareMathOperator{\Hyp}{Hyp}
\DeclareMathOperator{\I}{i}
\DeclareMathOperator{\Area}{Area}

\newcommand{\calA}{{\mathcal A}}
\newcommand{\quasi}{{\sf K}}
\newcommand{\radius}{{\sf r}}
\newcommand{\error}{{\sf c}}

\newcommand{\tY}{{\widetilde{Y}}}

\newcommand{\param}{{\mathchoice{\mkern1mu\mbox{\raise2.2pt
\hbox{$\centerdot$}}
\mkern1mu}{\mkern1mu\mbox{\raise2.2pt\hbox{$\centerdot$}}\mkern1mu}{
\mkern1.5mu\centerdot\mkern1.5mu}{\mkern1.5mu\centerdot\mkern1.5mu}}}
\long\def\@savemarbox#1#2{\global\setbox#1\vtop{\hsize\marginparwidth
  \@parboxrestore\tiny\raggedright #2}}
\newcommand{\finishproof}[1]{
  \def\FPArg{#1}
  \ifx\FPArg\Empty
        \newcommand\FPArg{\CrntSt}  \fi
  \smallbreak\noindent\makebox[\textwidth]{\hfill\fbox{\FPArg}}
  \medbreak\noindent
}

\def\square{\hfill${\vcenter{\vbox{\hrule height.4pt \hbox{\vrule width.4pt
height7pt \kern7pt \vrule width.4pt} \hrule height.4pt}}}$}

\newtheorem{theorem}{Theorem}%[section]
\newtheorem{proposition}[theorem]{Proposition}
\newtheorem{corollary}[theorem]{Corollary}
\newtheorem{lemma}[theorem]{Lemma}

\theoremstyle{definition}
\newtheorem{definition}[theorem]{Definition}
\newtheorem{example}[theorem]{Example}

\theoremstyle{remark}

\theoremstyle{theorem}
\newtheorem{introthm}{Theorem}

\newcommand{\thmref}[1]{Theorem~\ref{#1}}
\newcommand{\propref}[1]{Proposition~\ref{#1}}
\newcommand{\secref}[1]{\S\ref{#1}}
\newcommand{\lemref}[1]{Lemma~\ref{#1}}
\newcommand{\corref}[1]{Corollary~\ref{#1}} 
\newcommand{\figref}[1]{Fig.~\ref{#1}}

\DeclareMathOperator{\Mod}{Mod}
\DeclareMathOperator{\Ext}{Ext}

\DeclareMathOperator{\diam}{diam}

\newcommand{\from}{\colon\thinspace}
\newcommand{\emul}{\overset{.}{\asymp}}
\newcommand{\gmul}{\overset{.}{\succ}}
\newcommand{\lmul}{\overset{.}{\prec}}
\newcommand{\eadd}{\overset{+}{\asymp}}

\newcommand{\calB}{{\mathcal B}}

\newcommand{\calY}{{\mathcal Y}}
\newcommand{\calZ}{{\mathcal Z}}
\newcommand{\calT}{{\mathcal T}}

\newcommand{\calG}{\mathcal{G}}

\newcommand{\MF}{\mathcal{MF}}

\newcommand{\Teich}{Teichm\"uller }

\newcommand\R{{\mathbb R}}

\newcommand{\ep}{\epsilon}

\begin{document}

\title{Length of a Curve Is Quasi-Convex Along a Teichm\"uller Geodesic}
\author{Anna Lenzhen and Kasra Rafi}
\date{\today}

\begin{abstract}
We show that for every simple closed curve $\alpha$, the extremal
length and the hyperbolic length of $\alpha$ are quasi-convex
functions along any \Teich geodesic. As a corollary, we conclude that,
in \Teich space equipped with the \Teich metric, balls are
quasi-convex.
\end{abstract}

\maketitle

\section{Introduction}

In this paper we examine how the extremal length and the hyperbolic
length of a measured lamination change along a \Teich geodesic. We prove that 
these lengths are quasi-convex functions of time. The convexity issues
in \Teich space equipped with \Teich metric are hard to approach and are 
largely unresolved. For example it is not known whether it is possible 
for the convex hall of three points in \Teich space to be the entire space. 
(This is an open question of Masur.)

Let $S$ be a surface of 
finite topological type. Denote the \Teich space of $S$ equipped with the \Teich 
metric by $\calT(S)$. For a Riemann surface $x$ and a measured lamination $\mu$, 
we denote the extremal length of $\mu$ in $x$ by $\Ext_x(\mu)$ and the hyperbolic
length of $\mu$ in $x$ by $\Hyp_x(\mu)$. 

\begin{introthm} \label{Thm:Main}
There exists a constant $\quasi$, such that for every measured lamination
$\mu$, any \Teich geodesic $\calG$ and points $x,y,z \in \calT(S)$ appearing 
in that order along $\calG$ we have
$$
\Ext_y(\mu) \leq \quasi \max \big( \Ext_x(\mu), \Ext_z(\mu) \big),
$$
and
$$
\Hyp_y(\mu) \leq \quasi \max \big( \Hyp_x(\mu), \Hyp_z(\mu) \big).
$$
\end{introthm}

In $\sec$ \secref{Sec:Examples}, we provide some examples showing
that the quasi-convexity is the strongest statement one can hope for:

\begin{introthm}
The hyperbolic length and the extremal length of a curve 
are in general not convex functions of time along a \Teich geodesic. 
\end{introthm}

This contrasts with the results of  Kerckhoff \cite{kerckhoff:NR}, 
Wolpert \cite{wolpert:LN} and Bestvina-Bromberg-Fujiwara-Souto 
\cite{souto:SC}. They proved, respectively, that the 
hyperbolic length functions are convex along earthquake paths,
Weil-Petersson geodesics and a certain shearing paths. 

As a consequence of \thmref{Thm:Main}, we show that balls 
in \Teich space are quasi-convex.

\begin{introthm}
There exists a constant $\error$ so that, for every $x \in \calT(S)$,
every radius $\radius$ and points $y$ and $z$  in the ball
$\calB (x,\radius)$, the geodesic segment
$[y,z]$ connecting $y$ to $z$ is contained in $\calB(x,{\radius+\error})$.
\end{introthm}

We also construct an example of a long geodesic that stays near the boundary
of a ball, suggestion that balls in $\calT(S)$ may not be convex. \medskip

A \Teich geodesic can be described very explicitly as a deformation of a flat 
structure on $S$, namely, by stretching the horizontal direction and contracting 
the vertical direction. Much is known about the behavior of a \Teich
geodesic.  Our proof consists of combining the length estimates
given in \cite{minsky:PR, rafi:TT, rafi:TL} with the descriptions
of the behavior of a \Teich geodesic developed in 
\cite{rafi:SC, rafi:CM, rafi:LT}. 

As a first step, for a curve $\gamma$ and a quadratic differential 
$q$, we provide an estimate for the extremal length of $\gamma$ in the 
underlying conformal structure of $q$
(\thmref{Thm:LengthEstimate}) by describing 
what are the contributions to the extremal length of $\gamma$
from the restriction of $\gamma$ to various pieces of 
the flat surface associated to $q$. These pieces are either
\emph{thick sub-srufaces} or annuli with large moduli. 
We then introduce the notions of  \emph{essentially horizontal} and  
\emph{essentially vertical} (\corref{Cor:Significant} and 
Definition~\ref{Def:Essential}). Roughly speaking, a curve $\gamma$ is 
essentially horizontal in $q$ if the restriction of $\gamma$ to some 
piece of $q$ contributes a definite portion of the total extremal length
of $\gamma$ and if $\gamma$ is \emph{mostly horizontal} in that piece. 
We show that, while $\gamma$ is essentially vertical, its extremal length is 
\emph{essentially decreasing} and while $\gamma$ is essentially horizontal 
its extremal length is \emph{essentially increasing} (\thmref{Thm:Extremal}). 
This is because the flat length of the portion of $\gamma$ that is
mostly horizontal grows exponentially fast and becomes more and more 
horizontal. The difficulty with making this argument work is that
the thick-thin decomposition of $q$ changes as time goes by
and the portion of $\gamma$ that is horizontal and has a significant
extremal length can spread onto several thick pieces. That is why
we need to talk about the contribution to the extremal length of
$\gamma$ from every sub-arc of $\gamma$ (\lemref{Lem:Arcs}). 
The Theorem then follows from careful analysis of various possible
situations.  The proof for the hyperbolic length follows a similar path and is 
presented in $\sec$ \secref{Sec:Hyp}.

\subsection{Notation}
The notation $A \emul B$ means that the ratio $A/B$ is bounded both
above and below by constants depending on the topology of $S$ only.
When this is true we say $A$ is \emph{comparable} with $B$ or
$A$ and $B$ are comparable. 
The notation $A\lmul B$ means that $A/B$ is bounded above
by a constant depending on the topology of $S$.

\section{Background}
\subsection{Hyperbolic metric}
Let $x$ be a Riemann surface or equivalently (using uniformization) a complete 
hyperbolic metric on $S$.  By a \emph{curve} on $S$ we always mean a free 
homotopy class of non-trivial non-peripheral simple closed curve. Every 
curve $\gamma$ has a unique geodesic representative in the hyperbolic
metric $x$ which we call the \emph{$x$--geodesic representative}
of $\gamma$. We denote the hyperbolic length of the $x$--geodesic representative 
of $\gamma$ by $\ell_x(\gamma)$ and refer to it as the $x$--length $\gamma$. 

For a small positive constant $\ep_1$, the 
thick-think decomposition of $x$ is a pair $(\calA, \calY)$, where $\calA$ is the 
set of curves in $x$ that have hyperbolic length less than $\ep_1$ and $\calY$ is 
the set of components of $S \setminus (\cup_{\alpha \in \calA} \alpha)$. 
Note that, so far, we are only recording the topological information. One can make
this to a geometric decomposition as follows: for each $\alpha \in \calA$,
consider the annulus that is a regular neighborhood of the $x$--geodesic 
representative of $\alpha$ and has boundary length of $\ep_0$.
For $\ep_0 >\ep_1>0$ small enough, these annuli are disjoint 
(the Margulis Lemma) and their
complement is a union of subsurfaces with horocycle boundaries of length
$\ep_0$. For each $Y \in \calY$ we denote this representative of 
the homotopy class of $Y$ by $Y_x$.

If $\mu$ is a set of curves, then $\ell_x(\mu)$ is  the sum of the lengths
of the $x$--geodesic representatives of the curves in $\mu$. A short marking 
in $Y_x$ is a set $\mu_Y$ of curves in $Y$ so that $\ell_x(\mu_Y) =O(1)$ and
$\mu_Y$ fills the surface $Y$ (that is, every curve intersecting $Y$ 
intersects some curve in $\mu_Y$). 

If $\gamma$ is a curve and $Y \in \calY$, the \emph{restriction} 
$\gamma|_{Y_x}$ of $\gamma$ to $Y_x$ is the union of arcs
obtained by taking the intersection of the $x$--geodesic representative of
$\gamma$ with $Y_x$. Let $\gamma|_Y$ be the set of homotopy classes
(rel $\partial Y$) of arcs in $Y$ with end points on $\partial Y$. 
We think of $\gamma|_Y$ as a set of weighted arcs to  keep track of multiplicity. 
Note that $\gamma|_Y$ has only topological information while $\gamma|_{Y_x}$ 
is a set of geodesic arcs. An alternate way of defining $\gamma|_Y$ is to
consider the cover $\tY \to S$ corresponding to $Y$; that is, the cover
where $\tY$ is homeomorphic to $Y$ and such that $\pi_1(\tY)$
projects to a subgroup of $\pi_1(S)$ that is conjugate to $\pi_1(Y)$. 
Use the hyperbolic metric to construct a boundary at infinity for $\tY$. 
Then $\gamma|_{\tY}$ is the homotopy class of arcs in $\tY$ 
that are lifts of $\gamma$ and are not boundary parallel. Now the natural
homeomorphism from $\tY$ to $Y$ sends $\gamma|_{\tY}$
to $\gamma|_Y$. 

By $\ell_x(\gamma|_Y)$, we mean the $x$--length
of the shortest representatives of $\gamma|_Y$ in $Y_x$. 
It is well known that (see, for example, \cite{rafi:TL})
\begin{equation} \label{Eq:Hyp-Intersection}
 \ell_x(\gamma|_Y)  = \ell_x(\gamma|_{Y_x})\emul \I(\gamma, \mu_Y),
\end{equation}
where $\I(\param, \param)$ is the geometric intersection number and
$\I(\gamma, \mu_Y)$ is the sum of the geometric intersection numbers
between $\gamma$ and curves in $\mu_Y$.

\subsection*{Euclidean metric}
Let $q$ be a quadratic differential on $x$. In a local 
coordinate $z$, $q$ can be represented as $q(z) dz^2$ where $q(z)$
is holomorphic (when $x$ has punctures, $q$ is allowed to have poles of 
degree one at punctures). We call the metric $|q|=|q(z)| (dx^2 + dy^2)$ 
the flat structure of $q$. This is a locally flat metric
with singularities at zeros of $q(z)$ (see \cite{strebel:QD} for 
an introduction to the geometry of $q$). The $q$--geodesic representative
of a curve is not always unique; there may be a family of parallel copies 
of geodesics foliating a flat cylinder. For a curve $\alpha$, we denote
this flat cylinder of all $q$--geodesic representatives of $\alpha$ by
$F_\alpha^{\,q}$ or $F_\alpha$ if $q$ is fixed.  

Consider again the thick-thin decomposition $(\calA, \calY)$ of $x$.
(If $q$ is a quadratic differential on $x$, we sometimes call this the 
thick-thin decomposition of $q$. Note that  $(\calA, \calY)$ depends only on the
underlying conformal structure.) For $Y \in \calY$, the homotopy class of $Y$ 
has a representative with $q$--geodesic boundaries that is disjoint from the interior 
of the flat cylinders 
$F_\alpha$, for every $\alpha \in \calA$. We denote this subsurface by $Y_q$.
Note that $Y_q$ may be degenerate and have no interior (see \cite{rafi:SC}
for a more careful discussion). Let $\diam_q(Y)$ denote the
$q$--diameter of $Y_q$. We recall the following theorem
relating the hyperbolic and flat length of a curve in $Y$. 

\begin{theorem}[\cite{rafi:TT}] \label{Thm:q-Length}
For every curve $\gamma$ in $Y$
$$
\ell_q(\gamma) \emul \ell_x(\gamma) \, \diam_q(Y).
$$
\end{theorem}

Since $Y_q$ can be degenerate, one has to be more careful in defining 
$\ell_q(\gamma|_Y)$. Again we consider the cover $\tY \to S$ corresponding
to $Y$ and this time we equip $\tY$ with the locally Euclidean metric $\tilde q$
that is the pullback of $q$. The subsurface $Y_q$ lifts isometrically to
a subsurface $\tY_q$ in $\tY$. Consider the lift $\tilde \gamma$ of the 
$q$--geodesic representative of $\gamma$ to $\tY$ and the restriction
of $\tilde \gamma$ to $Y_q$.  
We define $\ell_q(\gamma|_Y)$ to be the $\tilde q$--length of this 
restriction. Note that $\ell_q(\gamma|_Y)$ may equal zero.  
(See the example at the end of \cite{rafi:TT}.)
However, a modified version of Equation~\eqref{Eq:Hyp-Intersection} 
still holds true for $\ell_q(\gamma|_Y)$:

\begin{proposition} \label{Prop:q--length}
For every curve $\gamma$ in $Y$
$$
\frac{\ell_q(\gamma|_Y)}{\diam_q(Y)} + \I (\gamma, \partial Y) 
       \emul \I(\gamma, \mu_Y).
$$
\end{proposition}

\begin{proof}
As above, consider the cover $\tY \to S$, the subsurface $\tY_q$ that is
the isometric lift of $Y_q$ and the lift $\tilde \gamma$ of the
$q$--geodesic representative of $\gamma$.  For every curve 
$\alpha \in \mu_Y$, there is a lift of $\alpha$ that is a simple closed curve. 
To simplify notation, we denote this lift again by $\alpha$ and the collection 
of these curves by $\mu_Y$. Let $d= \diam_q(Y)$, let 
$Z$ be the $d$--neighborhood of $\tY_q$ in $\tY$ and let $\omega$ be an arc 
in $Z$ constructed as follows: Choose and arc of $\gamma|_{\tilde Y_q}$ 
(which is potentially just a point) and at each end point $p$, extend this
arc perpendicular to $\partial {\tilde Y_q}$ until it hits $\partial Z$ at a point $p'$
(see \figref{Fig:Cover}).

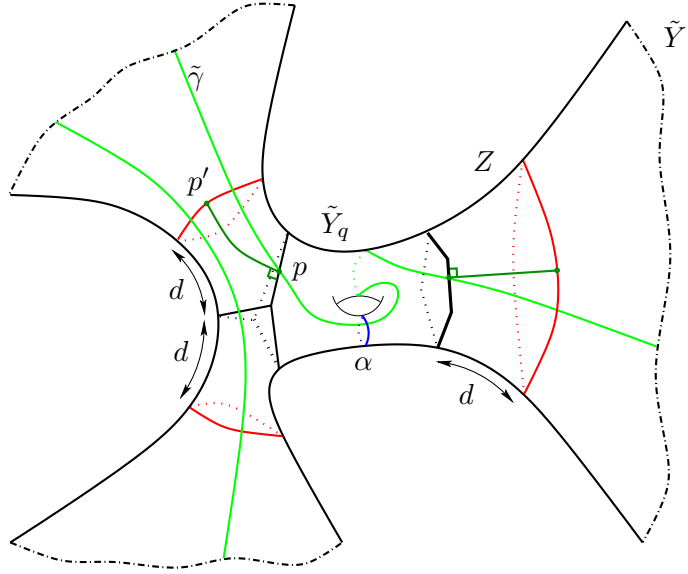
\begin{figure}[ht] \label{Fig:Cover}
\setlength{\unitlength}{0.01\linewidth}
\begin{picture}(70,58)
\put(0,0){\includegraphics[width=70\unitlength]{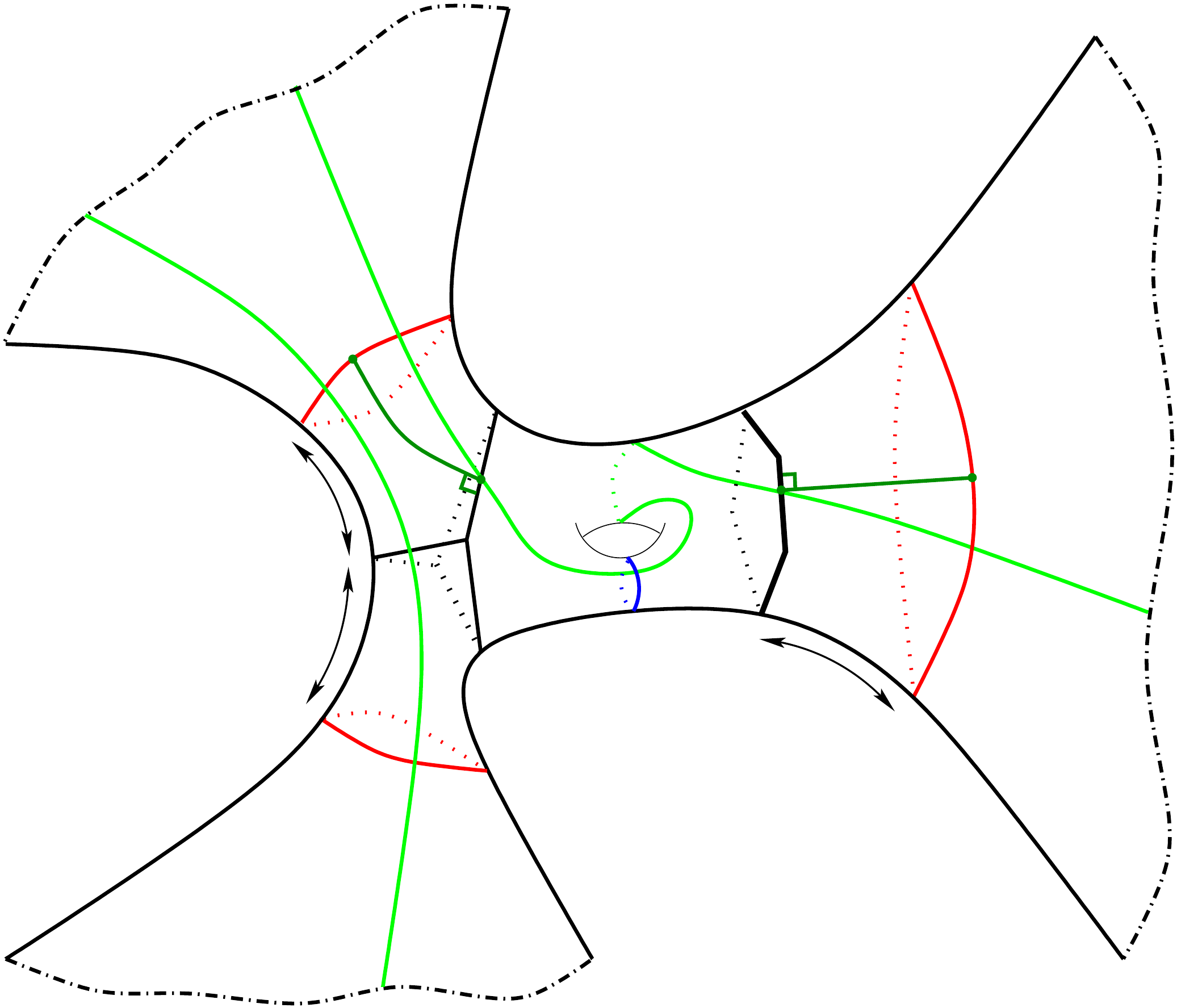}}
   \put(17.5,22){$d$}
   \put(17,28.5){$d$}
   \put(47.5,17.5){$d$}   
   \put(36.5,21){$\alpha$}
   \put(33,35.5){$\tilde Y_q$}
   \put(69,55){$\tilde Y$}
   \put(19,40){$p'$}
   \put(30,31){$p$}
   \put(49,42){$Z$}                  
   \put(19,51){$\tilde \gamma$}   
\end{picture}
 \caption{An arc in $\gamma|_{\tilde Y_{q}}$ can be extended to 
 an arc with end points on $Z$.}
\end{figure}

From the construction we have
$$
\ell_q(\omega) = \ell_q(\omega|_{\tilde Y_q}) + 2d .
$$
Summing over all such arcs, we have: 
$$
\sum_\omega \ell_q(\omega) 
  = \ell_q(\gamma|_Y) + d \I (\gamma, \partial Y).
$$
Also,
$$
\sum_\omega \I(\omega, \mu_Y) = \I(\gamma, \mu_Y).
$$
Hence, to prove the lemma, we only need to show 
\begin{equation} \label{Eq:Sufficient}
d \I(\omega, \mu_Y) \emul  \ell_q(\omega).
\end{equation}

The arguments needed here are fairly standard. In the interest
of brevity, we point the reader to some references instead of repeating
the arguments. Let $\alpha$ be a curve in $\mu_Y$. By \thmref{Thm:q-Length}, 
the $q$--length of the shortest essential curve in $Z$ (which has  hyperbolic 
length comparable with $1$) is comparable with $d$, 
hence the argument in the proof of \cite[Lemma 5]{rafi:TT} also implies
$$
\ell_q(\alpha) \, \ell_q(\omega) \gmul d^2 \I(\omega, \alpha).
$$
Therefore, $l_q(\omega) \gmul d \I(\omega, \alpha)$.  Summing over 
curves $\alpha \in \mu_Y$ (the number of which depends
on the topology of $S$ only), we have
$$
\ell_q(\omega) \gmul d \I(\omega, \mu_Y).
$$

It remains to show the other direction of Equation~\eqref{Eq:Sufficient}.
Here, one needs to construct  paths in $Y_q$ (traveling along
the geodesics in $\mu_Y$) representing arcs in $\gamma|_Y$ whose
lengths are of order $d \I(\gamma, \mu_Y)$. This is done
in the proof of  \cite[Theorem 6]{rafi:TT}).
\end{proof}

\subsection*{Regular and expanding annuli}
Let $(\calA, \calY)$ be the thick-thin decomposition of $q$ and let
$\alpha \in \calA$. Consider the $q$--geodesic representative of $\alpha$
and the family of regular neighborhoods of this geodesic in $q$.
Denote the largest regular neighborhood that is still homeomorphic to an 
annulus by $A_\alpha$. The annulus $A_\alpha$ contains the flat
cylinder $F_\alpha$ in the middle and two \emph{expanding} annuli 
on each end which we denote by $E_\alpha$ and $G_\alpha$:
$$
A_\alpha = E_\alpha \cup F_\alpha \cup G_\alpha.
$$
We call $E_\alpha$ and $G_\alpha$ expanding because if one considers the 
foliation of these annuli by curves that are equidistant to the geodesic representative 
of $\alpha$, the length of these curves increases as one moves away from
the $q$--geodesic representative of $\alpha$. This is in contrast with
$F_\alpha$ where all the equidistance curves have the same length. 
(See \cite{minsky:HM} for precise definition and discussion.) 
We denote the $q$--distance between the boundaries of $A_\alpha$ by 
$d_q(\alpha)$ and $q$--distance between the boundaries of $E_\alpha$,
$F_\alpha$ and $G_\alpha$ by $e_q(\alpha)$, $f_q(\alpha)$ and $g_q(\alpha)$ 
respectively. When $\alpha$ and $q$ are fixed, we simply use $e, f$ and $g$. 

\begin{lemma}[\cite{rafi:LT}] \label{Lem:EFE}
For $\alpha \in \calA$, 
$$
\frac 1 {\Ext_x(\alpha)} \emul \Mod_x(E_\alpha) +\Mod_x(F_\alpha)+\Mod_x(G_\alpha).
$$
Furthermore,  
$$
\Mod_x(E_\alpha) \emul \log \left( \frac {e}{\ell_q(\alpha)} \right), \qquad
\Mod_x(G_\alpha) \emul \log  \left( \frac {g}{\ell_q(\alpha)} \right),
$$
and
$$
\Mod_x(F_\alpha) \emul \frac {f}{\ell_q(\alpha)}.
$$
\end{lemma}

Let $\gamma$ be a curve. The restriction 
$\gamma|_{A_\alpha}$ is the set of arcs obtained from restricting the 
$q$--geodesic representative of $\gamma$ to $A_\alpha$, and 
$\ell_q(\gamma|_{A_\alpha})$ is the sum of the $q$--lengths of these
curves. 

\begin{lemma} \label{Lem:Restriction}
For the thick-thin decomposition $(\calA, \calY)$ of $q$, we have
$$
\ell_q(\gamma) \emul \sum_{Y \in \calY} \ell_q(\gamma|_Y)
+ \sum_{\alpha \in \calA} \ell_q(\gamma|_{A_\alpha}).
$$
\end{lemma}

\begin{proof}
The annuli $A_\alpha$ are not necessarily disjoint. But, the size of $\calA$ is 
uniformly bounded and $\ell_q(\gamma) \geq \ell_q(\gamma|_{A_\alpha})$. 
Similarly, the size of $\calY$ is uniformly bounded and 
$\ell_q(\gamma) \geq \ell_q(\gamma|_Y)$. Hence 
\begin{equation} \label{Eq:Annuli-Sum}
\ell_q(\gamma) \gmul  \sum_{Y \in \calY} \ell_q(\gamma|_Y) + 
\sum_{\alpha \in \calA} \ell_q(\gamma|_{A_\alpha}).
\end{equation}

To see the inequality in the other direction, we note that every segment
in the $q$--geodesic representative of $\gamma$ is either contained in some
$A_\alpha$, $\alpha \in \calA$ or in some $Y_q$, $Y \in \calY$. 
\end{proof}

\subsection{\Teich geodesics}
Let $q=q(z)dz^2$ be a quadratic differential on $x$. It is more
convenient to use the \emph{natural parameter} $\zeta=\xi+i\eta$, 
which is defined as
$$\zeta(w)=\int_{z_0}^{w}\sqrt{q(z)}\, dz.$$
In these coordinates, we have $q=d\zeta^2$. 
The lines $\xi=const$ with transverse measure $|d\xi|$ define the 
\textit{vertical} measured foliation, associated to $q$. Similarly, 
the \textit{horizontal} measured foliation is defined by $\eta=const$  and 
$|d\eta|$.  The transverse measure of an arc $\alpha$ with respect to 
$|d\xi|$, denoted by $h_q(\alpha)$, is called the \textit{horizontal length} of 
$\alpha$. Similarly, the \textit{vertical length} $v_q(\alpha)$ is the measure 
of $\alpha$ with respect to $|d\eta|$.

A \Teich geodesic can be described as follows. Given a Riemann surface
$x$ and a quadratic differential $q$ on $x$, we can obtain a
$1$--parameter family of quadratic differentials $q_t$ from $q$ 
so that, for $t\in \mathbb R$, if $\zeta=\xi+i\eta$ are natural coordinates for $q$, 
then $\zeta_t=e^{-t} \xi+i e^t \eta$ are natural coordinates for $q_t$. 
Let  $x_t$ be the conformal structure associated to $q_t$. Then 
$\calG:\mathbb R\to \calT(S)$ which sends $t$ to $x_t$, is a \Teich geodesic.

Let $\calG \from [a,b] \to \calT(S)$ be a \Teich geodesic and $q_a$ and $q_b$
be the initial and terminal quadratic differentials.  We use $\ell_{a}(\param)$ for 
$q_a$--length of a curve, we use $\Ext_a(\param)$ for the extremal length of a 
curve in $q_a$. Similarly, we denote by $\Mod_a(\param)$  the modulus of an 
annulus in $q_a$. We denote the thick thin decomposition of $q_a$ by 
$(\calA_a, \calY_b)$. We also write $e_a(\alpha), d_a(\alpha)$, $f_a(\alpha)$ 
and $\ell_a(\alpha)$ in place of $e_{q_a}(\alpha), d_{q_a}(\alpha)$, 
$f_{q_a}(\alpha)$ and $\ell_{q_a}(\alpha)$. When 
the curve $\alpha$ is fixed, we simplify notation even further and use $e_a$, 
$d_a$, $f_a$ and $\ell_a$. Also, we denote the flat annulus and the
expanding annuli corresponding to $\alpha$ in $q_a$ by
$F_\alpha^a$, $E_\alpha^a$ and $G_\alpha^a$, or by $F^a$, $E^a$, and
$G^a$ when $\alpha$ is fixed.
Similar notation applies to $q_b$. The following technical statement will be 
useful later.

\begin{corollary} \label{Cor:Ext/Length}
Let $\alpha$ be a curve in the intersection of $\calA_a$ and $\calA_b$.
Then 
$$
\frac{\Ext_a(\alpha)}{ \ell_{a}(\alpha)} \lmul 
e^{(b-a)} \frac{\Ext_b(\alpha)}{ \ell_b(\alpha)} 
$$
\end{corollary}

\begin{proof}
The length of an arc along a \Teich geodesic changes at most exponentially fast.
That is, $e^{b-a}$ is and upper-bound for $\frac{e_b}{e_a}$,
$\frac{f_b}{f_a}$, $\frac{g_b}{g_a}$ and 
$\frac{\ell_b}{\ell_a}$. Let $k=\frac{\ell_b}{\ell_a}$.
Then
$$
\frac{\ell_b\Mod_b(E^b)}{\ell_a\Mod_a(E^a)} \emul 
k \, \frac{\log\frac{e_b}{\ell_b}}{\log\frac{e_a}{\ell_a}} \leq 
k \, \frac{\log \left( {\frac{e^{b-a}}k \frac{e_a}{\ell_a}}\right)}{\log\frac{e_a}{\ell_a}}\leq 
k \frac{ \frac{e^{b-a}}k \log \frac{e_a}{\ell_a} }{\log\frac{e_a}{\ell_a}}\leq e^{b-a}.
$$
By a similar argument, 
$$\frac{\ell_b\Mod_b(G^b)}{\ell_a\Mod_a(G^a)}\lmul e^{b-a}$$
We also have 
$$
\frac{\ell_b\Mod_b(F^b)}{\ell_a\Mod_a(F^a)}\emul \frac{f_b}{f_a}\leq e^{b-a}.
$$
Then, by \lemref{Lem:EFE} and the estimates above, 
$$
\frac{\Ext_a}{\ell_a}\div \frac{\Ext_b}{\ell_b}\emul 
\frac{\ell_b\big( \Mod_b(E^b)+\Mod_b(F^b)+\Mod_b(G^b)\big)}
       {\ell_a\big( \Mod_a(E^a)+\Mod_a(F^a)+\Mod_a(G^a)\big)}\lmul e^{b-a},
$$
which is the desired inequality. 
\end{proof}

\subsection{Twisting}
 In this section we define several notions of twisting and discuss 
how they relate to each other. First, consider an annulus $A$
with core curve $\alpha$ and let $\tilde \beta$ and $\tilde \gamma$ be homotopy 
classes of arcs connecting the boundaries of $A$ (here, homotopy is relative to 
the end points of an arc).  The relative twisting of $\tilde \beta$ and $\tilde \gamma$ 
around $\alpha$, $\twist_\alpha(\tilde \beta, \tilde \gamma )$, is defined to be the 
geometric intersection number between $\tilde \beta$ and $\tilde \gamma$. 
If $\alpha$  is a curve on a surface $S$ and $\beta$ and $\gamma$ are two 
transverse curves to $\alpha$ we lift $\beta$ and $\gamma$ to the annular 
cover $\tilde S_\alpha$ of $S$ corresponding to $\alpha$. The curve $\beta$ 
(resp., $\gamma$ ) has at least one lift $\tilde \beta$ (resp., $\tilde \gamma$) 
that connects the boundaries of $\tilde S_\alpha$.  We define  $\twist_\alpha 
(\beta, \gamma)$ to be $\twist(\tilde \beta, \tilde \gamma)$. 
This is well defined up to a small additive error (\cite[\S 3]{minsky:PR}).

When the surface $S$ is equipped with a metric, one can ask
how many times does the geodesic representative of $\gamma$
twist around a curve $\alpha$. However, this needs to be made precise. 
When $x$ is a Riemann surface we define $\twist_\alpha(x, \gamma)$
to be equal to $\twist_\alpha(\beta, \gamma)$ where $\beta$
is the shortest geodesic in $x$ intersecting $\alpha$. For 
a quadratic differential $q$, the definition is slightly different. 
We first consider $F_\alpha$ and let $\beta$ be an arc connecting the 
boundaries of $F_\alpha$ that is perpendicular to the boundaries. 
We then define $\twist_\alpha(q, \gamma)$ to be the geometric
intersection number between $\beta$ and $\gamma|_{F_\alpha}$.
These two notions of twisting are related as follows:

\begin{theorem}[Theorem 4.3 in \cite{rafi:CM}] \label{Thm:TwoTwistings}
Let $q$ be a quadratic differential in the conformal class of $x$, and let 
$\alpha$ and $\gamma$ be two intersecting curves, then
$$
\big |\twist_\alpha (q, \gamma) - \twist_\alpha(x, \gamma) \big| 
  \lmul \frac{1}{\Ext_x(\alpha)}.
$$
\end{theorem}

%\begin{lemma}[Minsky] \label{Lem:Regular-Annulus}
%There is a regular annulus with significant modulus. 
%\end{lemma}

\section{An Estimate for the Extremal Length} \label{Sec:Extremal}
In \cite{minsky:PR}, Minsky has shown that the extremal length of a curve is comparable 
to the maximum of the contributions to the extremal length from the pieces 
of the thick-thin decomposition of the surface. Using this fact and some results 
in \cite{rafi:TT} and \cite{rafi:CM} we can state a similar result relating the 
flat length of a curve $\gamma$ to its extremal length. 

\begin{theorem} \label{Thm:LengthEstimate}
For a quadratic differential $q$ on a Riemann surface $x$, the corresponding 
thick-thin decomposition $(\calA, \calY)$ and a curve $\gamma$ on $x$,
we have
$$
\Ext_x(\gamma) \emul 
\sum_{Y \in \calY} \frac{\ell_q(\gamma|_Y)^2}{\diam_q(Y)^2} + 
\sum_{\alpha \in \calA} \left( \frac 1{\Ext_x(\alpha)}
     + \twist_\alpha^2(q,\gamma)\Ext_x(\alpha)  \right) \I(\alpha, \gamma)^2.
$$
\end{theorem}

\begin{proof}
First we recall \cite[Theorem 5.1]{minsky:PR} where Minsky states
that the extremal length of a curve $\gamma$ in $x$ is the maximum
of the contributions to the extremal length from each thick subsurface
and from crossing each short curve. The contribution from each curve 
$\alpha \in \calA$ is given by an expression \cite[Equation (4.3)]{minsky:PR} 
involving the $\I(\alpha, \gamma)$, $\twist_\alpha(x,\gamma)$ and
$\Ext_x(\alpha)$. For each subsurface $Y\in \calY$, the contribution 
to the extremal length from $\gamma|_Y$ is shown to be  
\cite[Theorem 4.3]{minsky:PR} the square of the hyperbolic length of $\gamma$ 
restricted to a representative of $Y$ with a horocycle boundaries of a fixed length
in $x$. This is known to be comparable to the square of the intersection 
number of $\gamma$ with a short marking $\mu_Y$ for $Y$.  

To be more precise, let $\mu_Y$ be a set of curves in $Y$ that fill $Y$
so that $\ell_x(\mu_Y) =O(1)$. Then, Minsky's estimate can be written as
\begin{equation} \label{Eq:Minsky}
\begin{split}
\Ext_x(\gamma) \emul 
 &\sum_{Y \in \calY} \I(\gamma, \mu_Y)^2+  \\
 &\sum_{\alpha \in \calA} \left( \frac{1}{\Ext_x(\alpha)}
     + \twist_\alpha^2(x,\gamma)\Ext_x(\alpha)  \right) \I(\alpha, \gamma)^2.
\end{split}     
\end{equation}
From \thmref{Thm:TwoTwistings},
$$
\big|\twist_\alpha(x,\gamma) - \twist_\alpha(q, \gamma) \big| =
O\left(\frac 1{\Ext_x(\alpha)}\right),
$$
and hence, 
$$
1+ \twist_\alpha(x,\gamma)\Ext_x(\alpha) \emul 
1+ \twist_\alpha(q,\gamma)\Ext_x(\alpha).
$$
Squaring both sides, and using $(a+b)^2 \emul a^2 + b^2$, we get
$$
1 + \twist_\alpha^2(x,\gamma)\Ext_x(\alpha)^2 \emul 
1+ \twist_\alpha^2(q,\gamma)\Ext_x(\alpha)^2.
$$
We know divide both sides by $\Ext_x(\alpha)$ to obtain
$$
\left(\frac{1}{\Ext_x(\alpha)} + \twist_\alpha^2(x,\gamma)\Ext_x(\alpha) \right)
\emul 
\left( \frac{1}{\Ext_x(\alpha)} + \twist_\alpha^2(q,\gamma)\Ext_x(\alpha) \right).
$$
That is, the second sum in Minsky's estimate is comparable to
the second sum in the statement of our Proposition. 

Now consider  the inequality in \propref{Prop:q--length}
for every $Y \in \calY$. After taking the square and adding up we get
$$
\sum_{Y \in \calY}  \frac{\ell_q(\gamma|_Y)^2}{\diam_q(Y)^2} 
+ \sum_{\alpha \in \calA}  \I(\gamma, \alpha)^2 \emul
\sum_{Y \in \calY} \I(\gamma, \mu_Y)^2
$$
But, the term $\sum_{\alpha \in \calA}  \I(\gamma, \alpha)^2$ is insignificant
compared with the term $  \frac{\I(\gamma, \alpha)^2}{\Ext_x(\alpha)}$
appearing in right side of Equation~ \eqref{Eq:Minsky}. Therefore, we can
replace the term $\sum_{Y \in \calY} \I(\gamma, \mu_Y)^2$ in
Equation~ \eqref{Eq:Minsky} with
 $\sum_{Y \in \calY}  \frac{\ell_q(\gamma|_Y)^2}{\diam_q(Y)^2}$
and obtain the desired inequality. 
\end{proof}

%\begin{remark} \label{Rem:ThreeHoled}
%In the first sum in above Proposition one need not include the components 
%that are thee holed spheres, since they do not contribute to extremal
%length. However, the theorem is true as stated and it is more useful for
%our purposes. 
%\end{remark}

To simplify the situation, one can provide a lower bound for extremal length
using the $q$--length of $\gamma$ and the sizes of the subsurface
$Y_q$, $Y \in \calY$, and $A_\alpha$, $\alpha \in \calA$. 

\begin{corollary} \label{Cor:Length-Only}
For any curve $\gamma$, the contribution to the extremal length of 
$\gamma$ from $A_\alpha$, $\alpha \in \calA$, is bounded
below by $\frac{ \ell_q(\gamma|_{A_\alpha}) ^2}{d_q(\alpha)^2}$.
In other words, 
$$
\Ext_x(\gamma) \gmul
\sum_{Y \in \calY} \frac{\ell_q(\gamma|_Y)^2}{\diam_q(Y)^2} + 
\sum_{\alpha \in \calA}  \frac{ \ell_q(\gamma|_{A_\alpha})^2 }{d_q(\alpha)^2}.
$$
\end{corollary}
\begin{proof}
Recall the notations $E_\alpha$, $F_\alpha$, $G_\alpha$, $e$
$f$ and $g$ from the background section. Denote the $q$--length
of $\alpha$ by $a$. Every arc of $\gamma|_{A_\alpha}$ has to cross
$A_\alpha$ and twist around $\alpha$, $\twist_\alpha(q, \gamma)$--times.
Hence, its length is less than $d_q(\alpha) + \twist_\alpha(q, \gamma) a$. Therefore,
$$
\ell_q(\gamma|_{A_\alpha})^2 \lmul 
\I(\alpha, \gamma)^2 \big( d_q(\alpha)^2   + \twist^2_\alpha(q, \gamma) a^2\big).
$$
Thus
\begin{align*}
\left( \frac{ \ell_q(\gamma|_{A_\alpha}) }{d_q(\alpha)\I(\alpha, \gamma))}\right)^2 
 & \emul \frac{d_q(\alpha)^{\,2} +  \twist_\alpha^2(q, \gamma) a^2 }{d_q(\alpha)^2}  
     \emul 1 +  \frac{\twist_\alpha^2(q, \gamma) } {d_q(\alpha)^{\,2}/a^2} \\
 & \lmul \frac 1{\Ext_x(\alpha)}
     + \frac{\twist_\alpha^2(q,\gamma)}
       {\log \frac {e}{a} +  \frac {f}{a} +  \log \frac {g}{a}} \\
  & \emul   \frac{1}{\Ext_x(\alpha)}   + \twist_\alpha^2(x,\gamma)\Ext_x(\alpha)  
\end{align*}
We now multiply both sides by $\I^2(\alpha, \gamma)$ and replace the second
term of estimate in \thmref{Thm:LengthEstimate} to obtain the corollary. 

The estimate here seems excessively generous, but there is a case
where the two estimates are comparable. This happens when $\alpha$ is not
very short, the twisting parameter is zero and $\gamma|_{A_\alpha}$ is a set
of $\I(\gamma, \alpha)$--many arcs of length comparable to one. 
\end{proof}

\subsection*{Essentially horizontal curves}
Roughly speaking, a curve $\gamma$ is essentially horizontal if there is an element in the thick-thin decomposition where $\gamma$ is mostly horizontal, and which contributes a definite portion of the extremal length of $\gamma$. One can always find a piece of a surface which  satisfies the latter, according to the following corollary.

\begin{corollary} \label{Cor:Significant}
Let  $(\calA, \calY)$ be a thick-thin decomposition for $q$
and let $\gamma$ be a curve that is not in $\calA$. Then
\begin{enumerate}
\item For every $Y \in \calY$
$$
\Ext_x(\gamma) \gmul \frac{\ell_q(\gamma|_Y)^2}{\diam_q(Y)^2}.
$$
\item For  $\alpha \in \calA$ and a flat annulus $F_{\alpha}$ whose core curve is $\alpha$,
$$ 
\Ext_x(\gamma) \gmul
\frac{\ell_q(\gamma|_{F_\alpha})^2 \Ext_x(\alpha)}{\ell_q(\alpha)^2}.
$$
\item For $\alpha \in \calA$ and an expanding annulus 
$E_\alpha$ whose core curve is $\alpha$, 
$$ 
\Ext_x(\gamma) \gmul \I(\alpha, \gamma)^2 \Mod_x(E_\alpha).
$$
\end{enumerate}
Furthermore, at least one of these inequalities is an equality up to a multiplicative error. 
\end{corollary}

\begin{proof}
The parts one and three follow immediately from \thmref{Thm:LengthEstimate}.
We prove part two. For $\alpha \in \calA$, let $a= \ell_q(\alpha)$
and let $f=f_q(\alpha)$.
As before, 
\begin{equation} \label{Eq:Triangle}
\ell_q(\gamma|_{F_\alpha})^2
\lmul (\twist_\alpha(q, \gamma)^2 a^2 + f^2)\I(\alpha, \gamma)^2.
\end{equation}
Hence
\begin{align*}
\frac{\ell_q(\gamma|_{F_\alpha})^2 \Ext_x(\alpha)}{\ell_q(\alpha)^2} 
& \lmul \frac{\twist_\alpha(q, \gamma)^2 a^2 + f^2}{a^2}
     \Ext_x(\alpha) \I(\alpha, \gamma)^2 \\
& \lmul \twist_\alpha(q, \gamma)^2 \Ext_x(\alpha) \I(\alpha, \gamma)^2 
+ \Ext_x(\alpha)\Mod_x(F_\alpha)^2\I(\alpha, \gamma)^2.
\end{align*}
But  $\Ext_x(\alpha)\Mod_x(F_\alpha)^2 \leq \frac 1{\Ext_x(\alpha)}$
and thus, by \thmref{Thm:LengthEstimate}, the above expression is bounded 
above by a multiple of $\Ext_x(\gamma)$.

To see that one of the inequalities have to be an equality, we observe that the 
number of pieces in the thick-thin decomposition $(\calA, \calY)$ is uniformly 
bounded. Therefore, some term in \thmref{Thm:LengthEstimate} is 
comparable to $\Ext_x(\gamma)$. If this is a term in the first sum then 
the inequality in part one is an equality. Assume for $\alpha \in \calA$ that
$$
\Ext_x(\gamma) \emul \frac{\I(\alpha, \gamma)^2}{\Ext_x(\alpha)}.
$$
We either have $\Ext_x(E_\alpha) \emul \Ext_x(\alpha)$ or 
$\Ext_x(F_\alpha) \emul \Ext_x(\alpha)$. In the first case,
the estimate in part three is comparable to $\Ext_x(\gamma)$. In the second case, 
\begin{align*}
\Ext_x(\gamma) &\emul \frac{\I(\alpha, \gamma)^2}{\Ext_x(F_\alpha)} 
\lmul \left(\frac{\I (\alpha, \gamma)^2 f^2 }{a^2} \right)
   \left( \frac{a}{f} \right)
   \lmul \frac{\ell_q(\gamma|_{F_\alpha})^2}{\ell_q(\alpha)} \Ext_x(\alpha),
\end{align*}
which means the inequality in part two is an equality. 

The only remaining case is when 
$$
\Ext_x(\gamma) \emul 
\twist_\alpha(q, \gamma)^2\Ext_x(\alpha) \I(\alpha, \gamma)^2.
$$
In this case, the estimate in part two is comparable to $\Ext_x(\gamma)$.
This follows from 
$\ell_q(\gamma|_{F_\alpha}) 
\gmul \twist_\alpha(q, \gamma)\ell_q(\alpha) \I(\alpha, \gamma)$.
\end{proof}

\begin{definition} \label{Def:Essential}
%We call  $Y$, $F_\alpha$ or $E_\alpha$ as above \emph{the subsurface with 
%a significant contribution}. We say $\gamma$ is \emph{essentially horizontal}, 
%if there is a subsurface with a significant contribution such that the
%restriction $\gamma$ to this subsurface is mostly horizontal (i.e., the
%its horizontal length is larger than its vertical length). 
We say that $\gamma$ is \emph{essentially horizontal}, 
if at least one of the following holds
\begin{enumerate}
\item $\Ext_x(\gamma) \emul \frac{\ell_q(\gamma|_Y)^2}{\diam_q(Y)^2}$ and $\gamma|_Y$ is mostly horizontal  (i.e., the
its horizontal length is larger than its vertical length) for some $Y \in \calY$.\\
\item $\Ext_x(\gamma) \emul \frac{\ell_q(\gamma|_{F_\alpha})^2 \Ext_x(\alpha)}{\ell_q(\alpha)^2}$ and $\gamma|_{F_\alpha}$ is mostly horizontal for some flat annulus  $F_\alpha$ whose core curve is $\alpha\in \calA$.\\
\item $ \Ext_x(\gamma)\emul  \I(\alpha, \gamma)^2 \Mod_x(E_\alpha)$ for some expanding annulus 
$E_\alpha$ whose core curve is $\alpha\in \calA$.\\
\end{enumerate}

\end{definition}

\subsection*{Extremal length of a geodesic arcs} \label{sec:X}
Consider the $q$--geodesic representative of a curve $\gamma$ and
let $\omega$ be an arc of this geodesic. We would like to estimate the 
contribution that $\omega$ makes to the extremal length of $\gamma$ in $q$. 
Let $(\calA, \calY)$ be the thick-thin decomposition of $q$. Let 
$\lambda_\omega$ be the maximum over $\diam_q(Y)$ for
subsurfaces $Y \in \calY$ that $\omega$ intersects and over all
$d_q(\alpha)$ for curves $\alpha \in \calA$ that $\omega$ crosses. 
Let $\sigma_\omega$ be the $q$--length of the shortest curve $\beta$
that $\omega$ intersects. We claim the contribution from $\omega$
to the extremal length of $\gamma$ is at least 
$$
\X(\omega) = \frac{\ell_q(\omega)^2}{\lambda_\omega^2} 
+ \log \frac{\lambda_\omega}{\sigma_\omega}.
$$
This is stated in the following lemma:
\begin{lemma} \label{Lem:Arcs}
Let $\Omega$ be a set of disjoints sub-arcs of $\gamma$.  Then
$$
\Ext_q(\gamma) 
 \gmul  |\Omega|^2 \, \min_{\omega \in \Omega} \X(\omega).
$$
\end{lemma}

\begin{proof}
Let $(\calA, \calY)$ be the thick-thin decomposition of $q$. We have 
\begin{align}
\label{Eq:One} \Ext_x(\gamma) 
 & \geq \sum_{Y \in \calY} \frac{\ell_q(\gamma|_Y)^2}{\diam_q(Y)^2} + 
       \sum_{\alpha \in \calA}  \frac{ \ell_q(\gamma|_{A_\alpha})^2}{d_q(\alpha)^2} \\
\label{Eq:Two} & \gmul \left(  
     \sum_{Y \in \calY} \frac{\ell_q(\gamma|_Y)}{\diam_q(Y)} + 
     \sum_{\alpha \in \calA}  \frac{\ell_q(\gamma|_{A_\alpha})}{d_q(\alpha)} \right)^2 \\
\label{Eq:Three} & \geq \left(  \sum_{Y \in \calY} \sum_{\omega \in \Omega} 
        \frac{\ell_q(\omega|_Y)}{\diam_q(Y)} + 
   \sum_{\alpha \in \calA} \sum_{\omega \in \Omega} 
    \frac{\ell_q(\omega|_{A_\alpha})}{d_q(\alpha)} \right)^2 \\
\label{Eq:Four} & \gmul \left(  \sum_{\omega \in \Omega} 
\left( \sum_{Y \in \calY}  \frac{\ell_q(\omega|_Y)}{\lambda_\omega} +
       \sum_{\alpha \in \calA}  \frac{\ell_q(\omega|_{A_\alpha})}{\lambda_\omega}
\right) \right)^2 
 \geq \left( \sum_{\omega \in \Omega} \frac{\ell_q(\omega)}{\lambda_\omega} \right)^2.
\end{align}
Inequality \eqref{Eq:One} follows from \corref{Cor:Length-Only}.
To obtain \eqref{Eq:Two}, we are using
$$
\sum_{i=1}^n x_i^2 \geq \frac 1n \Big(\sum_{i=1}^n x_i \Big)^2
$$
and the fact that the number of components in $\calY$ and in $\calA$ are uniformly
bounded. Line \eqref{Eq:Three} follows from the fact that arcs in $\Omega$
are disjoint sub-arcs of $\gamma$. To get \eqref{Eq:Four}, we then rearrange 
terms and use the fact that for all non-zero terms, $\diam_q(Y)$ and $d_q(\alpha)$ 
are less than $\lambda_\omega$.

Now let $\alpha_1, \ldots, \alpha_k$ be the sequence of curves in $\calA$ that 
$\omega$ intersects as it travels from the shortest curve $\beta$ to the largest 
subsurface it intersects, which has  size of at most $\lambda_\omega$. Note 
that either $\alpha_1 = \beta$ or $\alpha_1$ is the boundary of 
the thick subsurface containing $\beta_\omega$. Either way, 
$\ell_q(\alpha_1) \lmul \sigma_\omega$. Also, 
$d_q({\alpha_i}) \gmul \ell_q(\alpha_{i+1})$. This is because
$\alpha_i$ and $\alpha_{i+1}$ are boundaries of some subsurface
$Y \in \calY$. Finally, $d_q(\alpha_k) \gmul \lambda_\omega$. Therefore, 
$$
\sum_{i=1}^k \log \frac{d_q({\alpha_i})}{\ell_q(\alpha_i)} 
= \log \prod_{i=1}^k  \frac{d_q({\alpha_i})}{\ell_q(\alpha_i)} 
\gmul \log \frac{d_q({\alpha_k})}{\ell_q(\alpha_1)} 
\gmul \log \frac{\lambda_\omega}{\sigma_\omega}.
$$
Since $|\calA| \emul 1$, we can conclude that, for each $\omega$,
there is curve $\alpha$ so that $\I(\alpha, \omega) \geq 1$ and
$\displaystyle \log \frac{d_q(\alpha)}{\ell_q(\alpha)} 
\gmul \log\frac{\lambda_\omega}{\sigma_\omega}$.
Using \thmref{Thm:LengthEstimate}
\begin{align*}
\Ext_q(\gamma) &\gmul \sum_{\alpha \in \calA}
  \frac{\I(\alpha, \gamma)^2}{\Ext_q(\alpha)}
  \gmul \left( \sum_{\alpha \in \calA}  
  \frac{\I(\alpha, \gamma)}{\sqrt{\Ext_q(\alpha)}} \right)^2\\
&\gmul \left( \, \sum_{\omega \in \Omega} \sum_{\alpha \in \calA} 
\I(\alpha, \omega)  \sqrt{\log \frac{d_q(\alpha)}{\ell_q(\alpha)}} \ \right)^2\\
& \gmul \left( \, \sum_{\omega \in \Omega} 
  \sqrt{\log\frac{\lambda_\omega}{\sigma_\omega}} \ \right)^2.
\end{align*}

Combining the above two inequalities, we get
\begin{align*}
\Ext_q(\gamma) &\gmul  
\left( \, \sum_{\omega \in \Omega} \frac{\ell_q(\omega)}{\lambda_\omega} \right)^2 
+  \left( \, \sum_{\omega \in \Omega}  
\sqrt{\log\frac{\lambda_\omega}{\sigma_\omega}} \ \right)^2\\
&\gmul \left(  \sum_{\omega \in \Omega} \sqrt{X(\omega)} \right)^2
\geq |\Omega|^2 \min_\omega X(\omega). \qedhere
\end{align*}
\end{proof}

We also need the following technical lemma.
\begin{lemma} \label{Lem:Subsurface}
Let $q_a$ and $q_b$ be two points along a \Teich geodesic and let
$(\calA_a, \calY_a)$ and $(\calA_b, \calY_b)$ be their thick-thin
decompositions respectively. Let $Y \in \calY_a$, 
\begin{itemize}
\item If $\beta \in \calA_b$ intersects $Y$, then $d_b(\beta) \leq e^{(b-a)}\diam_a(Y)$.
\item If $Z \in \calY_b$ intersects $Y$, then $\diam_b(Z) \leq e^{(b-a)}\diam_a(Y)$.
\end{itemize} 
Similarly,  if $\alpha \in \calA_a$, 
\begin{itemize}
\item If $\beta \in \calA_b$ intersects $\alpha$, then 
$d_b(\beta) \leq e^{(b-a)}\ell_{a}(\alpha)$.
\item If $Z \in \calY_b$ intersects $\alpha$, then 
$\diam_b(Z) \leq e^{(b-a)}\ell_{a}(\alpha)$.
\end{itemize} 
\end{lemma}

\begin{proof}
Let $\gamma$ be the shortest curve system in $q_a$ that fills $Y$. 
Then $l_{b}(\gamma)\lmul e^{(b-a)}\diam_a(Y)$. 
If $Y$ intersects $\beta \in \calA_b$ then some curve in $\gamma$ has to 
intersect  $A_\beta$ essentially and we have
$$d_b(\alpha) \leq l_{b}(\gamma) \lmul e^{(b-a)}\diam_a(Y).$$

If $Y$ intersects some subsurface $Z \in \calY_b$, then $Z$ has an essential 
arc $\omega$ in $Z$ whose $q_b$--length is less than the $q_b$--length
of $\gamma$. Also, if $Y$ intersects a boundary component $\delta$ of $Z$, 
$$
\ell_{b}(\gamma) \gmul d_b(\delta) \geq \ell_{b}(\delta). 
$$
By doing a surgery between $\omega$ and $\delta$, one obtains
an essential curve in $Z$ whose $q_b$--length is less than a fixed
multiple of $\gamma$. Hence
$$
\diam_b(Z) \lmul l_{b}(\gamma) \lmul e^{(b-a)}\diam_a(Y).
$$
Which is as we desired. The argument for $\alpha \in \calA_a$ is similar. 
\end{proof}

\section{The main theorem}
Let $\calG \from \R \to \calT(S)$ be a \Teich geodesic. We denote
the Riemann surface $\calG(t)$ by $\calG_t$ and the corresponding
quadratic differential in $\calG_t$ by $q_t$. For a curve $\gamma$,
denote the extremal length of $\gamma$ on $\calG_t$ by 
$\Ext_t(\gamma)$ and the thick-thin decomposition of
$q_t$ simply by $(\calA_t, \calY_t)$.

\begin{theorem} \label{Thm:Extremal}
There exists a constant $\quasi$, such that for every
measured foliation $\mu$, any \Teich geodesic $\calG$
and points $x,y,z \in \calT(S)$ appearing in that order along $\calG$
we have
$$
\Ext_x(\mu) \leq \quasi \max \big( \Ext_y(\mu), \Ext_z(\mu) \big).
$$
\end{theorem}

\begin{proof}
Let the times $a<b<c \in \R$ be such that $x=\calG_a$, $y=\calG_b$ and
$z=\calG_c$. Recall that the extremal length 
$$\Ext \from \MF(S) \times \calT(S) \to \R$$
is a continuous function, and that the weighted simple closed curves
are dense in $\MF(S)$. Since the limit of quasi-convex functions is 
itself quasi-convex and a multiple of a quasi-convex function is
also quasi-convex (with the same constant in both cases), it is sufficient 
to prove the theorem for simple closed curves only. That is, we can assume
that every leaf of $\mu$ is homotopic to a curve $\gamma$ and
the transverse measure is one.

If the extremal length of $\gamma$ is very short at $b$ but not very short 
at either $a$ or $c$, the statement is clearly true. If $\gamma$ is short at 
times $a$, $b$ and $c$, the statement is already known;
in \cite{rafi:CM}, it is shown that the interval where $\gamma$ is ``short"
is connected  \cite[Corollary 3.4]{rafi:CM} and along this interval
the extremal length (which comparable with hyperbolic length \cite{maskit:HE})
is quasi-decreasing until the balanced time and is quasi-increasing
afterwards \cite[Theorem 1.2]{rafi:CM}. Therefore, we can assume there is a lower 
bound on the length of $\gamma$ at $b$, where the lower bound depends on the topology of $x$ only. 

By \corref{Cor:Significant}, there is a subsurface of $q_b$ with significant 
contribution and the restriction of $\gamma$ to this subsurface is either mostly 
horizontal or mostly vertical. That is, $\gamma$ is either essentially horizontal or 
essentially vertical. If $\gamma$ is essentially horizontal, \propref{Prop:Horizontal} 
implies  $\Ext_b(\gamma) \lmul \Ext_c(\gamma)$ and we are done.
Otherwise, $\gamma$ is essentially vertical. In this case, we can reverse 
time, changing the role of the horizontal and vertical foliations, and using \propref{Prop:Horizontal} again conclude $\Ext_b(\gamma) \lmul \Ext_a(\gamma)$. 
This finishes the proof.
\end{proof}

\begin{proposition} \label{Prop:Horizontal}
If $\gamma$ is essentially horizontal for the quadratic differential $q_a$,
then for every $b>a$ we have
$$
\Ext_b(\gamma) \succ \Ext_a(\gamma).
$$
\end{proposition}

\begin{proof}
We argue in 3 cases according to which inequality in \corref{Cor:Significant}
is an equality up a multiplicative error. 

\subsection*{Case 1} Assume there is a subsurface $Y \in \calY_a$ so that
$$
\Ext_a(\gamma) \emul \frac{\ell_{a}(\gamma|_Y)^2}{\diam_a(Y)^2}
$$
such that $\gamma|_Y$ is mostly horizontal. We have
$ \ell_{b}(\gamma|_Y) \gmul e^{(b-a)}\ell_{a}(\gamma|_Y)$.
Let $\calZ$ be the set of subsurfaces in $\calY_b$ that intersect 
$Y$ and let $\calB$ be a set of annuli $A_\alpha$, where $\alpha \in \calA_b$  
and $\alpha$ intersects $Y$. Then $Y_b$ is contained in 
the union of $ \bigcup_{Z \in \calZ} Z_b $ and $ \bigcup_{\beta \in \calB} A_b(\beta).$
Therefore, 
$$
\ell_{b}(\gamma|_Y) \leq \sum_{Z \in \calZ} \ell_{b}(\gamma|_{Z})+ 
\sum_{\beta \in \calB} \ell_{b}(\gamma|_{A_b(\beta)}).
$$
We also know that
$$
\diam_b(Z) \leq  e^{(b-a)} \diam_a(Y) \qquad\text{and}\qquad
d_b({\beta}) \leq  e^{(b-a)} \diam_a(Y).
$$
Therefore, 
\begin{align*}
\Ext_b(\gamma) 
& \gmul \sum_{Z \in \calZ} \frac{\ell_{b}(\gamma|_{Z})^2}{\diam_b(Z)^2}+ 
\sum_{\beta \in \calB} \frac{\ell_{b}(\gamma|_{A_b(\beta)})^2}{(d_b(\beta))^2}\\
& \gmul \frac{ \sum_{Z \in \calZ} \ell_{b}(\gamma|_{Z})^2+ 
\sum_{\beta \in \calB} \ell_{b}(\gamma|_{A_b(\beta)})^2}{e^{2(b-a)} \diam_a(Y)^2}\\
 & \gmul \left(\frac{e^{(b-a)}\ell_{a}(\gamma|_Y)}{ e^{(b-a)}\diam_a(Y)}\right)^2
  \emul \Ext_a(\gamma). \\
\end{align*}
 
\subsection*{Case 2} Assume that there is a curve $\alpha \in \calA$ so that
$$ 
\Ext_a(\gamma) \emul 
\frac{\ell_{a}(\gamma|_{F_\alpha})^2 \Ext_a(\alpha)}{\ell_{a}(\alpha)^2},
$$
and $\gamma|_{F_\alpha}$ is mostly horizontal. Then
$\ell_{b}(\gamma|_{F_\alpha}) \gmul e^{(b-a)} \ell_{a}(\gamma|_{F_\alpha})$.
If $\alpha$ is still short in $q_b$ then the proposition follows from
\corref{Cor:Ext/Length}. 

Otherwise, let $\calZ$ be the set of sub-surfaces in $\calY_b$ that intersect  
$\alpha$ and let $\calB$ be the set of curves in $\calA_b$ that intersect
$\alpha$. Since $F_\alpha$ has geodesic boundaries, it is contained 
in the union of  $ \bigcup_{Z \in \calZ} Z_b $ and 
$\bigcup_{\beta \in \calB} A_b(\beta).$ The rest of the proof is exatly
as in the previous case with the additional observation that
$\Ext_b(\alpha) \geq \ep_1 \geq \Ext_a(\alpha)$.

\subsection*{Case 3} Assume there an expanding annulus $E$
with large modulus and the core curve $\alpha$ such that 
$$\Ext_a(\gamma) \emul i(\alpha, \gamma)^2 \Mod_a(E)$$
and $\gamma|_E$ is mostly horizontal. 
Let $\Omega$ be the set of sub-arcs of $\gamma$ that start and end
in $\alpha$ and whose restriction to $E$ is at least $(1/3)$--horizontal
(that is, the ratio of the horizontal length to the vertical length is bigger than 
$1/3$--times). We have 
$$|\Omega| \geq (1/4) \, i(\alpha, \gamma).$$
Otherwise, $(3/4)$ of arc are $3$-vertical, which implies that the 
total vertical length is larger than the total horizontal length.  Recall that 
$$\Mod_a(E) \emul \log(d_a(\alpha)/\ell_{a}(\alpha)).$$
For $\omega \in \Omega$ we have $\ell_{a}(\omega) \geq 2d_a(\alpha)$. Since,
the restriction of $\omega$ to $E$ is mostly horizontal, we have
$$\ell_{b}(\omega) \succ e^{(b-a)} d_a(\alpha).$$ 
The arc $\omega$ intersects $\alpha$, so 
$\sigma_\omega \leq \ell_{b}(\alpha)\leq e^{(b-a)}\ell_{a}(\alpha)$.
Therefore,
$$
X(\omega) \gmul \frac{ e^{2(b-a)} (d_a(\alpha))^2}{\lambda_\omega^2} 
+ \log \frac{\lambda_\omega}{e^{(b-a)}\ell_a(\alpha)}.
$$
This expression is minimum when $\lambda_\omega =  \sqrt{2}e^{(b-a)} d_a(\alpha)$. That is,
$$
X(\omega) \gmul \log \frac{d_a(\alpha)}{\ell_a(\alpha)} \emul \frac{1}{\Ext_a(\alpha)}.
$$
Using \lemref{Lem:Arcs} we have
$$
\Ext_b(\gamma) \succ |\Omega|^2 \min_{\omega \in \Omega} \X(\omega)
\gmul \frac{i(\alpha, \gamma)^2}{\Ext_b(\alpha)} \gmul \Ext_a(\gamma).
$$
This finishes the proof in this case.
\end{proof}

\section{Quasi-Convexity of a Ball in \Teich Space}

Consider a Riemann surface $x$. Let $\calB(x,\radius)$ denote the ball
of radius $\radius$ in $\calT(S)$ centered at $x$.

\begin{theorem} \label{Thm:QuasiConvex}
There exists a constant $\error$ such that, for ever $x \in \calT(S)$,
every radius $\radius$ and point $y$ and $z$  in the ball
$\calB (x,\radius)$, the geodesic segment
$[y,z]$ connecting $y$ to $z$ is contained in $\calB(x,{\radius+\error})$.
\end{theorem}

\begin{proof}
Let $u$ be a point on the segment $[y,z]$. It is sufficient to show that
$$
d_\calT(x,u) \leq \max \big( d_\calT(x,y) , d_\calT(x,z) \big)+c.
$$
There is a measured foliation $\mu$ such that
$$
d_\calT(x,u) = \frac 12 \log \frac{\Ext_u(\gamma)}{\Ext_x(\gamma)}.
$$
Also, from the convexity of extremal lengths \thmref{Thm:Extremal}
we have
$$
\Ext_u(\mu) \leq \quasi \max \big( \Ext_y(\mu), \Ext_z(\mu) \big)
$$
Therefore,
\begin{align*}
d_\calT(x,u) 
 &\leq \frac12 \log \left( \frac{ \quasi 
  \max \big( \Ext_y(\gamma), \Ext_z(\gamma) \big)}{\Ext_x(\gamma)} \right)\\
 &\leq  \max \big( d_\calT(x,y) , d_\calT(x,z) \big)+c. \qedhere
\end{align*}
\end{proof}
%%%%%%%%%%
%%%%%%%%%% Hyperbolic Length

\section{Quasi-convexity of hyperbolic Length} \label{Sec:Hyp}
In this section, we prove the analogue of \thmref{Thm:Extremal} for
the hyperbolic length:

\begin{theorem} \label{Thm:Hyperbolic}
There exists a constant $\quasi'$, such that for every 
measured foliation $\mu$, any \Teich geodesic $\calG$ and times 
$a < b < c \in \R$, we have
$$
\Hyp_b(\mu) \leq \quasi' \max \big( \Hyp_a(\mu), \Hyp_c(\mu) \big)
$$
\end{theorem}
\begin{proof}
The argument is identical to the one for \thmref{Thm:Extremal}, with Corollary \ref{Cor:Significant_Hyp} and Proposition \ref{Prop:Horizontal-Hyp} 
being the key ingredients. They are stated and proved below. 
\end{proof}
Our main goal for the rest of this section is the Proposition \ref{Prop:Horizontal-Hyp}.  
To make the 
reading easier, we often take note of the similarities and skip some arguments 
when they are nearly identical to those for the extremal length case. \medskip

In place of \thmref{Thm:LengthEstimate} we have
\begin{theorem}\label{Thm:LengthEstimate-Hyp}
For a quadratic differential $q$ on a Riemann surface $x$, the corresponding thick-thin decomposition $(\calA,\calY)$ and a curve $\gamma$ on $x$, we have
\begin{equation} \label{Eq:Hyp}
\begin{split}
\Hyp_x(\gamma) \emul
&\sum_{Y \in \calY} \frac{\ell_q(\gamma|_Y)}{\diam_q(Y)} + \\
&\sum_{\alpha\in \calA} \left[ \log \frac{1}{\Ext_x(\alpha)} 
     + \twist_\alpha(q,\gamma)\Ext_x(\alpha)  \right] i(\alpha, \gamma).
\end{split}     
\end{equation}
\end{theorem}
\begin{proof}
The hyperbolic length of a curve $\gamma$ is, up to a universal multiplicative 
constant, the sum of the lengths of $\gamma$ restricted to the pieces of the thick-thin decomposition of the surface. 
 The hyperbolic length of $\gamma|_Y$ is comparable to the intersection number of $\gamma$ with a short marking $\mu_Y$ of $Y$, which is, by  the  Proposition \ref{Prop:q--length}, up to a multiplicative error,  
$$ \frac{\ell_q(\gamma|_Y)}{\diam_q(Y)}+\sum_{\alpha\in \partial Y } i(\gamma,\alpha).$$
The contribution from each curve $\alpha\in\calA$ is 
(see, for example, \cite[Corollary 3.2]{rafi:LT}),
$$
\left[\log{\frac{1}{\Hyp_x(\alpha)}}+\Hyp_x(\alpha)\twist_\alpha(x,\gamma)\right] i(\alpha,\gamma).
$$
Thus, we can write an estimate for the hyperbolic length of $\gamma$ as 
\begin{equation} \nonumber
%\label{Eq:HypFirst}
\Hyp_x(\gamma) \emul
\sum_{Y \in \calY} \frac{\ell_q(\gamma|_Y)}{\diam_q(Y)} + 
\sum_{\alpha\in \calA} \left[\log{\frac{1}{\Hyp_x(\alpha)}}+\Hyp_x(\alpha)\twist_\alpha(x,\gamma)\right] i(\alpha,\gamma).
\end{equation}
Note that we are not adding 1 to the sum in the parenthesis above since the sum is actually substantially greater. \medskip

To finish the proof,  we need to replace $\Hyp_x(\alpha)$ with $\Ext_x(\alpha)$ and   $\twist_\alpha(x,\gamma)$ with $\twist_\alpha(q,\gamma)$.
 Maskit has shown \cite{maskit:HE} that, 
when $\Hyp_x(\alpha)$ is small, 
$$\frac{\Hyp_x(\alpha)}{\Ext_x(\alpha)}\emul 1.$$ 
Hence, we can replace $\Hyp_x(\alpha)$ with $\Ext_x(\alpha)$. 
Further, it follows from  \thmref{Thm:TwoTwistings} that 
$$  |\twist_\alpha(q,\gamma)\Ext_x(\alpha)- \twist_\alpha(x,\gamma)\Ext_x(\alpha)|=O(1).$$ 
 Since $\log \frac{1}{\Ext_x(\alpha)}$ is  at least 1 for $\alpha\in\calA$,
we have
$$ \log \frac{1}{\Ext_x(\alpha)}+\twist_\alpha(q,\gamma)\Ext_x(\alpha)\emul \log \frac{1}{\Ext_x(\alpha)}+  \twist_\alpha(x,\gamma)\Ext_x(\alpha),$$
which means that we can replace  $\twist_\alpha(x,\gamma)$ with $\twist_\alpha(q,\gamma)$.
\end{proof} 
We almost immediately have:
\begin{corollary} \label{Cor:Significant_Hyp}
Let  $(\calA, \calY)$ be a thick-thin decomposition for $q$
and let $\gamma$ be a curve that is not in $\calA$. Then
\begin{enumerate}
\item For every $Y \in \calY$
$$
\Hyp_x(\gamma) \gmul \frac{\ell_q(\gamma|_Y)}{\diam_q(Y)}.
$$
\item For every  $\alpha \in \calA$ and the flat annulus  $F_\alpha$ whose core curve is $\alpha$,
$$ 
\Hyp_x(\gamma) \gmul
\log\Mod_x(F_\alpha)i(\alpha,\gamma)
$$
\item For every $\alpha \in \calA$ 
$$
\Hyp_x(\gamma)\gmul \twist_\alpha(q,\gamma)\Ext_x(\alpha)i(\alpha,\gamma). 
$$
\item For every $\alpha \in \calA$ and an expanding annulus 
$E_\alpha$ whose core curve is $\alpha$, 
$$ 
\Hyp_x(\gamma) \gmul \log\Mod_x(E_{\alpha})i(\alpha,\gamma).
$$
\end{enumerate}
Furthermore, at least one of these inequalities is an equality up to a multiplicative error. 
\end{corollary}
\begin{proof}
The parts $(1)-(4)$ follow immediately from \thmref{Thm:LengthEstimate-Hyp} and the fact that 
the reciprocal of the extremal length of a curve $\alpha$ is bounded below by the modulus of any annulus homotopic to $\alpha$. 
Further, since the number of pieces in the thick-thin decomposition  $(\calA, \calY)$ is uniformly bounded, some term in \thmref{Thm:LengthEstimate-Hyp} has to be comparable with $\Hyp_x(\gamma)$. The only non-trivial case is when that term is $\log\frac{1}{\Ext_x(\alpha)}i(\alpha,\gamma)$ for some $\alpha\in \calA$. But by \lemref{Lem:EFE}, either 
$$\frac{1}{\Ext_x(\alpha)} \emul \Mod_x(F_\alpha),$$ 
or  
$$\frac{1}{\Ext_x(\alpha)} \emul \Mod_x(E_\alpha),$$
and the Lemma holds.
\end{proof}
As in the section \secref{Sec:Extremal}, we need a notion of 
\emph{essentially horizontal} for hyperbolic length. We say that $\gamma$ is \emph{essentially horizontal}, 
if at least one of the following holds
\begin{enumerate}
\item $\Hyp_x(\gamma) \emul \frac{\ell_q(\gamma|_Y)}{\diam_q(Y)}$ and $\gamma|_Y$ is mostly horizontal  (i.e., the
its horizontal length is larger than its vertical length) for some $Y \in \calY$.\\
\item $\Hyp_x(\gamma) \emul \log\Mod_x(F_\alpha)i(\alpha,\gamma)$ and $\gamma|_{F_\alpha}$ is mostly horizontal for some flat annulus  $F_\alpha$ whose core curve is $\alpha\in \calA$.\\
\item $\Hyp_x(\gamma) \emul \twist_\alpha(q,\gamma)\Ext_x(\alpha)i(\alpha,\gamma)$ and $\gamma|_{F_\alpha}$ is mostly horizontal for some flat annulus  $F_\alpha$ whose core curve is $\alpha\in \calA$.\\
\item $ \Hyp_x(\gamma)\emul \log\Mod_x(E_{\alpha})i(\alpha,\gamma)$ for some expanding annulus 
$E_\alpha$ whose core curve is $\alpha\in \calA$.\\
\end{enumerate}

Further, Corollary \ref{Cor:Length-Only} is replaced with
\begin{corollary} \label{Cor:Length-Only-Hyp}
For any curve $\gamma$, the contribution to the hyperbolic length of 
$\gamma$ from $A_\alpha$, $\alpha \in \calA$, is bounded
below by $\frac{ \ell_q(\gamma|_{A_\alpha})}{d_q(\alpha)}$.
In other words, 
$$
\Hyp_x(\gamma) \gmul
\sum_{Y \in \calY} \frac{\ell_q(\gamma|_Y)}{\diam_q(Y)} + 
\sum_{\alpha \in \calA}  \frac{ \ell_q(\gamma|_{A_\alpha}) }{d_q(\alpha)}.
$$
\end{corollary}
\begin{proof}
Identical to the proof of Corollary \ref{Cor:Length-Only} after removing the squares and taking $\log$ when necessary.
\end{proof}
Instead of the function $\X(\omega)$, to estimate the hyperbolic length of an
arc, we define
$$
\h(\omega) = 
\frac{\ell_q(\omega)}{\lambda_{\omega}}+
\log \max \left\{\log\frac{\lambda_{\omega}}{\sigma_{\omega}},1\right\}.
$$
In place of \lemref{Lem:Arcs} we get
\begin{lemma} \label{Lem:Arcs-Hyp}
Let $\Omega$ be a set of disjoints sub-arcs of $\gamma$.  Then
$$
\Hyp_x(\gamma) 
 \succ  |\Omega| \, \min_{\omega \in \Omega} \h(\omega).
$$
\end{lemma}
\begin{proof}
Identical to the proof of  \lemref{Lem:Arcs} after removing the squares and taking $\log$ when necessary.
\end{proof} 

Finally, we have the analog  of  Proposition \ref{Prop:Horizontal}. 
\begin{proposition} \label{Prop:Horizontal-Hyp}
If $\gamma$ is essentially horizontal for the quadratic differential $q_a$,
then for every $b>a$ we have
$$
\Hyp_b(\gamma) \succ \Hyp_a(\gamma).
$$
\end{proposition}
\begin{proof}
By the definition of essentially horizontal, there are four cases to consider. We deal with two of them, the flat annulus case and the twisting case, at once in Case 2. 
\subsection*{Case 1}  
There is a thick subsurface $Y$ where $\gamma$ is mostly horizontal and such that 
$$
\Hyp_a(\gamma) \emul \frac{\ell_{a}(\gamma|_Y)}{\diam_a(Y)}
$$
The proof is as in the extremal length case after removing the squares.

\subsection*{Case 2} There exists a curve $\alpha\in \calA$ so that 
\begin{equation} \label{flat}
\Hyp_a(\gamma)\emul \log\Mod_a(F_\alpha)i(\alpha,\gamma),
\end{equation} or 
\begin{equation}\label{twist}
\Hyp_a(\gamma)\emul \twist_\alpha(a,\gamma)\Ext_a(\alpha)i(\alpha,\gamma),
\end{equation}
and  $\gamma|_{F_\alpha}$ is mostly horizontal. We argue in three sub-cases.
 
\subsection*{Case 2.1} Suppose first that $\alpha$ is no longer short at $t=b$
and either \eqref{flat} or \eqref{twist} holds. Let $\calZ$ be the set of sub-surfaces in 
$\calY_b$ that intersect  
$\alpha$ and let $\calB$ be the set of curves in $\calA_b$ that intersect
$\alpha$. Then, by Corollary \ref{Cor:Length-Only-Hyp} and \lemref{Lem:Subsurface},
\begin{align*}
\Hyp_b(\gamma) &\gmul
  \sum_{Z \in \calZ} \frac{\ell_b(\gamma|_Z)}{\diam_b(Z)} + 
  \sum_{\beta \in \calB}  \frac{ \ell_b(\gamma|_{A_\beta}) }{d_b(\beta)}\\
&\gmul \sum_{Z \in \calZ} \frac{\ell_{b}(\gamma|_Z)}{e^{b-a}\ell_{a}(\alpha)} + 
  \sum_{\beta \in \calB}  \frac{ \ell_{b}(\gamma|_{A_\beta}) }{e^{b-a}\ell_{a}(\alpha)}.
\intertext{But $F_\alpha$ is contained in 
$\left( \bigcup_{Z \in \calZ} Z \right) \cup \left( \bigcup_{\beta \in \calB} A_\beta \right)$.}
&\gmul \frac{\ell_{b}(\gamma|_{F_\alpha})}{e^{b-a}\ell_{a}(\alpha)}
  \gmul \frac{\ell_{a}(\gamma|_{F_\alpha})}{\ell_{a}(\alpha)}\\
&\gmul \max \big\{\log\Mod_a(F_\alpha)i(\alpha,\gamma),  
                    \twist_\alpha(a,\gamma)\Ext_a(\alpha)i(\alpha,\gamma)\big\}\\
&\emul \Hyp_a(\gamma).
\end{align*}

%%%new
\subsection*{Case 2.2} Suppose now that $\alpha \in \calA_b$ and that
\eqref{flat} holds. If $\alpha$ is mostly vertical at time $a$, the
extremal
length of $\alpha$ is decreasing exponentially fast for some interval
$[a,c]$. That is, $\Mod_c(F_\alpha) \gmul \Mod_a(F_\alpha)$.
It is sufficient to show that for $b\geq c$,
$$
\Hyp_b(\gamma) \gmul \log\Mod_c(F_\alpha) \I(\alpha, \gamma).
$$

Our plan is to argue that, while the modulus of $F_{\alpha}$ is decreasing,
the hyperbolic length of $\gamma$ is not decreasing by much because the
curve is twisting very fast around $\alpha$. We need to estimate the twisting
of $\gamma$ around $\alpha$. Let $\omega$ be one of the arcs of
$\gamma|_{F_\alpha}$. Note that $\omega$ is mostly horizontal at $c$
(since it was at $a$) and its length is larger than $f_c(\alpha)$.
Also, since $\alpha$ is mostly horizontal
at $c$, $f_t(\alpha)$ is decreasing exponentially fast at $t=c$. Hence, after
replacing $c$ with a slightly larger constant, we can assume
$\omega$ is significantly larger than $f_a(\alpha)$ and therefore,
the number of times $\omega$ twists around $\alpha$ is approximately
the length ratio of $\omega$ and $\alpha$
(see Equation 15 and 16 in \cite{rafi:CM} and the related discussion for
more details). That is, for $c\leq t\leq b$, $\twist_\alpha(q_t,\gamma)$
is essentially constant:
$$
\twist_\alpha(q_t,\gamma)\emul \frac{\ell_{t}(\omega)}{\ell_{t}(\alpha)}
\emul  \frac{e^{(t-a)}\ell_{a}(\omega)}{e^{(t-a)}\ell_{a}(\alpha)}
= \frac{\ell_{a}(\omega)}{\ell_{a}(\alpha)}.
$$
Therefore,
$$
\Mod_c(F_\alpha) = \frac{f_c(\alpha)}{\ell_{c}(\alpha)}
\leq \frac{\ell_{a}(\omega)}{\ell_{a}(\alpha)} \emul \twist_\alpha(q_c,\gamma).
$$
Keeping in mind that, for $k\geq 0$, the function $f(x)=-\log{x}+kx >
\log{k}$ ,
we have
\begin{align*}
\Hyp_b(\gamma)
&\gmul \left[\log\frac{1}{\Ext_{b}(\alpha)}
    +\twist_\alpha(b,\gamma)\Ext_b(\alpha)\right]i(\alpha,\gamma)\\
&\gmul \log  \big( \twist_\alpha(q_b,\gamma) \big)\, i(\alpha,\gamma)
    \gmul \log \Mod_c(F_\alpha) \, i(\alpha,\gamma).
\end{align*}

\subsection*{Case 2.3} Suppose that $\alpha \in \calA_b$ and that \eqref{twist} holds. 
Since $\gamma$ crosses $\alpha$, $\Hyp_a(\gamma)$ is greater than
a large multiple of $\I(\alpha, \gamma)$. Hence
$$\Hyp_a(\gamma)\emul \twist_\alpha(a,\gamma)\Ext_a(\alpha)i(\alpha,\gamma)$$
implies that $\twist_\alpha(a,\gamma)$ is much larger than
$\Mod_a(F_\alpha)$. That is, the angle between $\gamma$ and $\alpha$ is
small. Therefore, after perhaps replacing $a$ with a slightly larger number, 
we can assume that $\alpha$ is mostly horizontal and that, for $a\leq t \leq b$,
\begin{equation} \label{twistbounds}
\twist_\alpha(t,\gamma)\emul \frac{\ell_t(\omega)}{\ell_t(\alpha)} 
\end{equation}
Applying \thmref{Thm:LengthEstimate-Hyp}, Equation~\eqref{twistbounds}, 
Corollary \ref{Cor:Ext/Length}, Equation~\eqref{twist} in that order
we obtain:
\begin{align*}
\Hyp_b(\gamma)
 &\gmul \left[\twist_\alpha(b,\gamma)\Ext_b(\alpha)\right]i(\alpha,\gamma)\\
 &\emul \frac {\Ext_b(\alpha)}{\ell_{b}(\alpha)}
                \ell_{b}(\omega) i(\alpha,\gamma)\\
&\gmul \frac{e^{b-a}}{e^{b-a}} \frac{\Ext_{a}(\alpha)}{\ell_a(\alpha)}
                \ell_{a}(\omega) i(\alpha,\gamma)\\
&\gmul \left[\twist_\alpha(a,\gamma)\Ext_a(\alpha)\right]i(\alpha,\gamma).
\end{align*}

\subsection*{Case 3} There is a curve $\alpha\in\calA$ with expanding annulus 
$E_{\alpha}$ such that $\gamma|_{E_\alpha}$ is mostly horizontal with 
$$
\Hyp_a(\gamma)\emul \log\Mod_a(E_{\alpha})i(\alpha,\gamma)
$$
Following the proof for the corresponding case for extremal length, we have
\begin{align*}
\h(\omega) &\gmul 
\frac{e^{b-a}d_a(\alpha)}{\lambda_{\omega}}+
\log\max\left\{\log\frac{\lambda_{\omega}}{e^{b-a}\ell_{a}(\alpha)},1\right\} \\
& \gmul \log\log \frac{d_a(\alpha)}{\ell_{a}(\alpha)}\emul \log \Mod_a(E_{\alpha}).
\end{align*}
One can verify the second inequality as follows. If 
$\frac{\lambda_{\omega}}{e^{b-a}}\leq \sqrt{d_a(\alpha) \ell_{a}(\alpha)}$, then 
$$
\frac{e^{b-a}d_a(\alpha)}{\lambda_{\omega}}
  \geq\sqrt\frac{d_a(\alpha)}{ \ell_{a}(\alpha)}
  \gmul \log\log \frac{d_a(\alpha)}{\ell_{a}(\alpha)}.
$$ 
Otherwise, 
$$
\log\log\frac{\lambda_{\omega}}{e^{b-a}\ell_{a}(\alpha)}
\geq \log\log\sqrt{\frac{d_a(\alpha)}{\ell_{a}(\alpha)}}
\gmul  \log\log \frac{d_a(\alpha)}{\ell_{a}(\alpha)},$$
and applying \lemref{Lem:Arcs-Hyp}, we have
$$
\Hyp_b(\gamma)
\gmul  |\Omega| \min_{\omega \in \Omega} \h(\omega)
\gmul i(\alpha,\gamma) \log\Mod_a(E_{\alpha})\emul \Hyp_a(\gamma).
$$
This finishes the proof.
\end{proof}

%%%%%%%%%%

\section{Examples} \label{Sec:Examples}
This section contains two examples. In the first example we
describe a \Teich geodesic and a curve whose length is not
convex along this geodesic. The second example is of a very long 
geodesic that spends its entire length near the boundary of a round ball. 

\begin{example}[Extremal length and hyperbolic length are not convex]
To prove that the extremal and the hyperbolic lengths are not convex, we 
construct a quadratic differential and analyze these two lengths for a specific 
curve along the geodesic associated to this quadratic differential. We show that  
on some interval the average slope (in both cases) is some positive number 
and on some later interval the average slope is near zero. This shows that the 
two length functions are not convex. Note that, scaling the weight of a curve 
by a factor $k$ increases the hyperbolic and the extremal length of that curve
by factors  of $k$ and $k^2$, respectively. Thus, after scaling, 
one can produce examples where the average slope is very large on some 
interval  and near zero on some later interval. 

Let $0< a \ll 1$. Let $T$ be rectangular torus obtained from identifying the opposite 
sides of the rectangle $[0,a]\times [0,\frac{1}{a}]$. Also, let $C$ be a euclidean 
cylinder obtained by identifying vertical sides of $[0,a]\times[0,a]$. Take two 
copies $T_1$ and $T_2$ of $T$, each cut along a horizontal segment of length 
$a/2$ (call it a slit), and join them by gluing $C$ to the slits. This defines a 
quadratic differential $q$ on a genus two surface $x_0$. The horizontal and 
the vertical trajectories of $q$ are those obtained from the horizontal and the vertical
foliation of $\R^2$ by lines parallel to the $x$--axis and $y$--axis respectively. 
We now consider the \Teich geodesic based at $x_0$ in the direction of $q$. 
Let $\alpha$ be a core curve of cylinder $C$. We will show that, for small
enough $a$, $\Ext_{x_t}(\alpha)$ and $\Hyp_{x_t}(\alpha)$ are not convex 
along $x_t$. 

Let $\rho$ be the metric which coincides with the flat metric of $q$ on $C$ and on 
the two horizontal bands in $T_i$ of width and height $a$ with the slit in the middle, and is 0 otherwise. 
The shortest curve in the homotopy class of $\alpha$ has length $a$ in this metric. 
Then we have 
\begin{equation}\label{ex:lowexp}
\Ext_{x_0}(\alpha)\geq \frac{a^2}{3a^2}=\frac13. 
 \end{equation}
Also, at time $t<0$, we have $\Mod_{x_{t}}(C)=\frac{ae^{-t}}{a e^t}= e^{-2t}$ 
and, therefore, 
\begin{equation}\label{ex:upexp}
 \Ext_{x_{t}}(\alpha)\leq e^{2t}.
\end{equation}
Hence we see that the extremal length of $\alpha$ grows exponentially on 
$(-\infty,0)$. In particular, the average slope on the interval 
$\mathcal J = (-2,0)$ is more than $\frac18$. 
By Proposition 1 and Corollary 3 in \cite{maskit:HE}, 
\begin{equation}\label{ex:maskit} 2
e^{-\frac{1}{2}\Hyp_{x}(\alpha)}\leq \frac{\Hyp_{x}(\alpha)}{\Ext_{x}(\alpha)}\leq\pi.
\end{equation}
and it is easy to see that the slope of $\Hyp_{x_{t}}(\alpha)$ on this interval is also 
greater than $\frac18$.

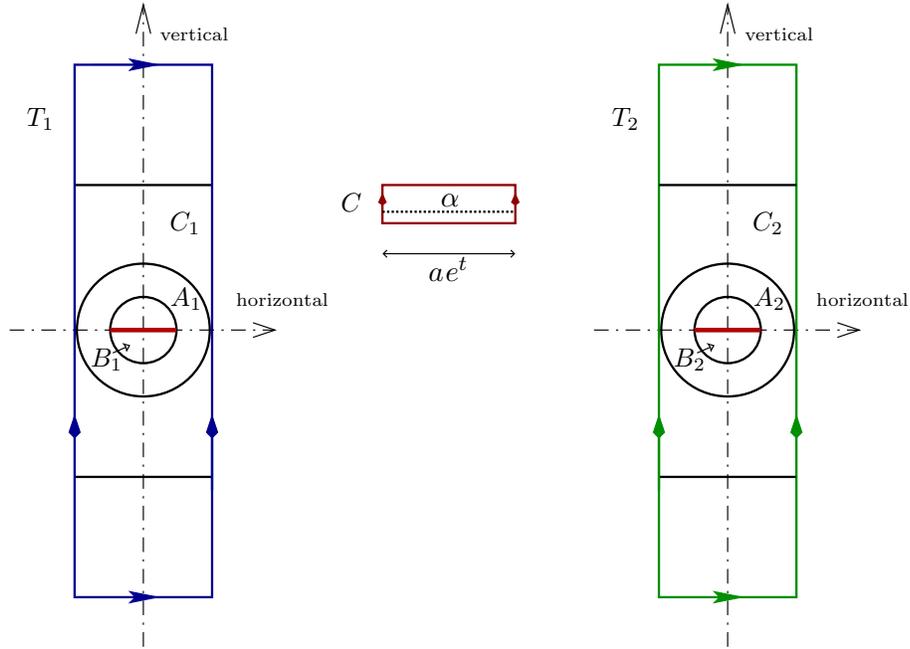
\begin{figure}[ht] 
\setlength{\unitlength}{0.01\linewidth}
\begin{picture}(90,75)
\put(0,0){\includegraphics[width=90\unitlength]{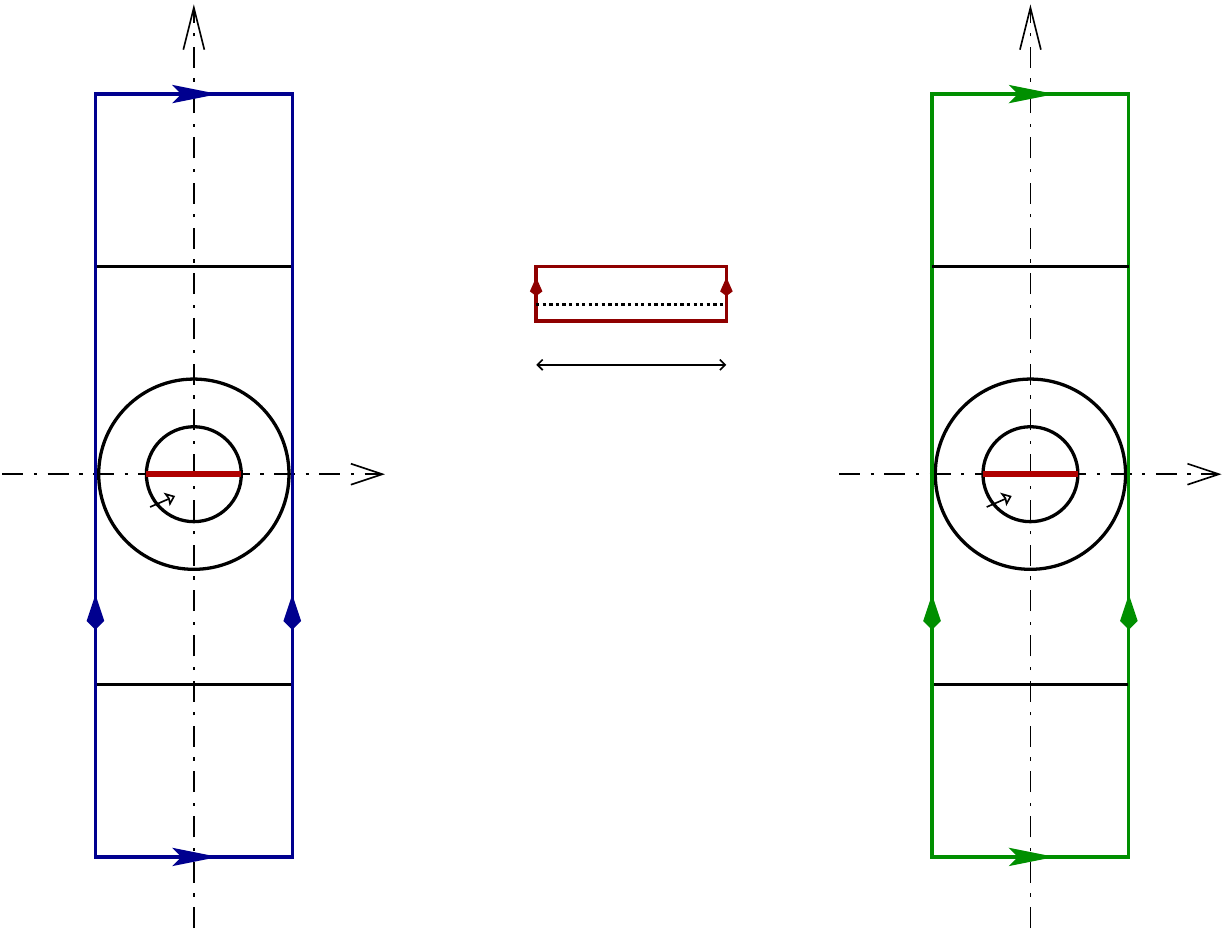}}
   \put(16,64){$\text{\tiny vertical}$}
   \put(77.5,64){$\text{\tiny vertical}$}
   \put(2,55){\small $T_1$} 
   \put(63.5,55){\small $T_2$}   
   \put(17,44){\small $C_1$}  
   \put(78.3,44){\small $C_2$} 
   \put(8.7,29.7){\small $B_1$}  
   \put(70,29.7){\small $B_2$}
   \put(17.1,36){\small $A_1$}  
   \put(78.4,36){\small $A_2$} 
   \put(45.5,46.5){$\alpha$}
   \put(24,36){$\text{\tiny horizontal}$}
   \put(85,36){$\text{\tiny horizontal}$}
   \put(35,46){\small $C$}
   \put(44.2,38.7){$a\:\!e^t$}
% \put(44.2,35){$\stackrel{\longleftrightarrow}{a\:\!e^t}$}
\end{picture}
 \caption{Metric $\rho_t$ on $x_t$ when $t>0$.} \label{ex1}
\end{figure}

Further along the ray, when $t>0$, the modulus of $C$ is decreasing exponentially.
We estimate  $\Ext_{x_t}(\alpha)$ for $t \in \mathcal I= (0,\frac{1}{2}\log\frac{1}{a^2})$.
For the lower bound, consider the cylinder $A$ which is the union of $C$ 
and the maximal annuli in $T_i$ whose boundary is a round circle centered
at the middle of the slit. Then $A$ contains two disjoint copies of annuli of inner radius 
$(ae^t/4)$ and outer radius $(a e^t/2)$ (the condition on $t$
guarantees that these annuli do not touch the top or the bottom edges of $T_1$ and 
$T_2$.) Both of these annuli have modulus of $\frac{1}{2\pi}\log 2$, and therefore 
$\Mod_{x_t}(A)\geq \frac{1}{\pi}\log 2$. Hence 
\begin{equation}\label{ex:lowconst}
\Ext_{x_t}(\alpha)\leq \frac{1}{\Mod_{x_t}(A)}\leq  \frac{\pi}{\log 2}.
\end{equation}

For the upper bound, we use the metric $\rho_t$ defined as follows 
(see \figref{ex1}):  Let $A_{i}$ be the annuli in $T_i$ centered at the midpoints 
of the corresponding slits with inner radius $\frac{a e^t+\delta}{4}$ and outer radius 
$\frac{ae^t-\delta}{2}$ for a very small $\delta$.  Let $\rho_t$ be 
$\frac{1}{|z|}|dz|$ on $A_i$, and the flat metric $|dz|$ on $C$  scaled so that 
the circumference is $2\pi$. The complement of $A_i$ and $C$ 
consists of two annuli $B_1,B_2$ and two once-holed tori $C_1$ and $C_2$
with $B_i,C_i\in T_i$. On each of these components, we will define $\rho_t$ so that 
the shortest representative of $\alpha$ has length at least $2\pi$ and the area 
is bounded above.  More precisely, let $\rho_t=\frac{2\pi}{a e^t}|dz|$ on $B_i$.
On $C_i$, let $\rho_t$ be $\frac{2}{ae^t-\delta}|dz|$ if 
$|Im\,  z|< \frac{1}{2}(\pi+1)(ae^t-\delta)$, and zero otherwise.

The area of $C$ in this metric is $(2\pi a e^t) \dot (2\pi a e^{-t}) = O(1)$.  
The pieces $B_i$ and $C_i$ have diameters of order $O(ae^t)$ in $\rho_t$ and hence have area of order $O(1)$.  The annuli $A_i$ in this metric are isometric to flat cylinders 
of circumference $2 \pi$ and width less than $\log 2$, which also has area one. 
Thus,
$$
\text{Area}_{\rho_t}(\mathcal S) = O(1).
$$

 Also, the $\rho_t$-length of any curve $\alpha'$ homotopic to $\alpha$ is $\ell_{\rho_t}(\alpha')\geq 2\pi$. Indeed, any curve contained in one of the annuli, has $\rho_t$-length at least $2\pi$. Morever, any subarc  of $\alpha'$ with endpoints on a boundary component of an annulus can be homotoped relative to the endpoints to the boundary without increasing the length. 
 
Since the area of $\rho_t$ is uniformly bounded above (independent of $a$ and $t$) 
and the length of $\alpha$ in $\rho_t$ is larger $2\pi$, the extremal length
\begin{equation}\label{ex:upconst}
   \Ext_{x_t}(\alpha) \geq 
   \frac{\underset{\alpha'\sim\alpha}{\inf} \ell_{\rho_t}(\alpha')^2} 
          {\Area_{\rho_t}(\mathcal S)},
\end{equation}
is bounded below on $\mathcal I$ by a constant independent of $t$ and $a$. 
Combining this with  \eqref{ex:upconst} we see that, as $a \to 0$ 
(and hence the size of $\mathcal I$ goes to $\infty$), the average slope of 
$\Ext_{x_t}(\alpha)$ on $\mathcal I$ is near zero. In particular, the average slope 
on $\mathcal I$  can be made smaller than $\frac 18$ which implies that
the function $\Ext_{x_t}(\alpha)$ is not convex. Combining \eqref{ex:maskit}  and the estimates of the  extremal length 
above, we come to the same conclusion about $\Hyp_{x_t}(\alpha)$.
\end{example}

\begin{example} [Geodesics near the boundary] 
Here we describe how, for any $R>0$, a geodesic segment of length 
comparable to $R$ can stay near the boundary of a ball of radius $R$. 
This example suggests that metric balls in $\calT(S)$ might not be convex.

Let $x$ be a point in the thick part of  $\calT(S)$ and $\mu_x$ be the short
marking of $x$. Pick any two disjoint curves $\alpha, \beta$ in $\mu_x$. 
Let $y=\mathcal{D}_{(\alpha)}^nx$, and $z=\mathcal {D}_{(\alpha,\beta)}^nx$, where 
$\mathcal{D}_{(*)}$ is the Dehn twist around a multicurve $(*)$. The intersection 
numbers between the short markings of $x,y,z$ satisfy
$$
 i(\mu_x,\mu_y)\emul  i(\mu_x,\mu_z)\emul i (\mu_y, \mu_z) \emul n.
$$
Hence, by Theorem 2.2 in \cite{rafi:TL}, we have
$$
d_\calT(x,y)\eadd d_\calT(x,z)\eadd d_\calT(y,z) \eadd \log{n}.
$$
That is, $[y,z]$ is a segment of length $\log n$ whose end points
are near the boundary of the ball $\calB(x, \log n)$.  
We will show, for $w\in[y,z]$, that $d_\calT(x, w) \eadd \log n$, which means
the entire geodesic $[y,z]$ stays near the boundary of the ball $\calB(x, \log n)$.  
Let $\alpha'$ be a curve that intersects $\alpha$, is disjoint from $\beta$
and $\Ext_y(\alpha')=O(1)$.  Since $\alpha'$ intersect $\alpha$,
$$\Ext_x(\alpha')\emul n^2,$$  
and since $\alpha'$ is disjoint from $\beta$, 
$$\Ext_z(\alpha') = O(1).$$
By Theorem \ref{Thm:Extremal}, 
$$
\Ext_w(\alpha')  \leq\quasi \max \big\{\Ext_y(\alpha'),\Ext_z(\alpha)\big\} =O(1). 
$$
We now have 
$$
d_{\calT(S)}(w,x)\geq\frac{1}{2}\log\frac{\Ext_x(\alpha')}{\Ext_w(\alpha')}\eadd \log{n}.
$$ 
\end{example}

\bibliographystyle{alpha}
\bibliography{../../main}

\begin{thebibliography}{BBFS09}

\bibitem[BBFS09]{souto:SC}
M.~Bestvina, K.~Bromberg, K.~Fujiwara, and J.~Souto.
\newblock Shearing coordinates and convexity of length functions on
  teichm\"ueller space.
\newblock arXiv:0902.0829, 2009.

\bibitem[CR05]{rafi:TL}
Y.~Choi and K.~Rafi.
\newblock {Comparison between Teichm\"uller and Lipschitz metrics}.
\newblock arXiv:math.GT/0510136, 2005.

\bibitem[CRS06]{rafi:LT}
Y.~Choi, K.~Rafi, and C.~Series.
\newblock {Lines of minima and Teichm\"uller geodesics}.
\newblock arXiv:math.GT/0605135, 2006.

\bibitem[Ker83]{kerckhoff:NR}
S.~P. Kerckhoff.
\newblock The {N}ielsen realization problem.
\newblock {\em Ann. of Math. (2)}, 117(2):235--265, 1983.

\bibitem[Mas85]{maskit:HE}
B.~Maskit.
\newblock Comparison of hyperbolic and extremal lengths.
\newblock {\em Ann. Acad. Sci. Fenn. Ser. A I Math.}, 10:381--386, 1985.

\bibitem[Min92]{minsky:HM}
Y.~N. Minsky.
\newblock Harmonic maps, length, and energy in {T}eichm\"uller space.
\newblock {\em J. Differential Geom.}, 35(1):151--217, 1992.

\bibitem[Min96]{minsky:PR}
Y.~N. Minsky.
\newblock Extremal length estimates and product regions in {T}eichm\"uller
  space.
\newblock {\em Duke Math. J.}, 83(2):249--286, 1996.

\bibitem[Raf05a]{rafi:SC}
K.~Rafi.
\newblock A characterization of short curves of a {T}eichm{\"u}ller geodesic.
\newblock {\em Geometry and Topology}, 9:179--202, 2005.

\bibitem[Raf05b]{rafi:TT}
K.~Rafi.
\newblock Thick-thin decomposition of quadratic differentials.
\newblock preprint, 2005.

\bibitem[Raf07]{rafi:CM}
K.~Rafi.
\newblock {A combinatorial model for the Teichm\"uller metric}.
\newblock {\em Geom. Funct. Anal.}, 17(3):936--959, 2007.

\bibitem[Str80]{strebel:QD}
K.~Strebel.
\newblock {\em Quadratic differentials}, volume~5 of {\em A series of Modern
  Surveys in Mathematics}.
\newblock Springer-Verlag, Berlin, 1980.

\bibitem[Wol87]{wolpert:LN}
S.~A. Wolpert.
\newblock Geodesic length functions and the {N}ielsen problem.
\newblock {\em J. Differential Geom.}, 25(2):275--296, 1987.

\end{thebibliography}

\end{document}